\documentclass[11pt,a4paper]{article}

\usepackage{amsmath, amsfonts, amssymb}
\usepackage{color}

\def\be#1\ee{\begin{equation}#1\end{equation}}

\newtheorem{thm}{Theorem}
\newtheorem{lem}[thm]{Lemma}
\newtheorem{prop}[thm]{Proposition}
\newtheorem{cor}[thm]{Corollary}

\newtheorem{rem}[thm]{Remark}

\DeclareMathOperator{\esssup}{esssup}

\def\P{{\mathbb{P}}}
\def\R{\mathbb{R}}
\def\E{\mathbb{E}\,}

\def\N{{\mathbb N}}
\def\dd{\mbox{d}}

\newenvironment{proof}[1][] {\noindent {\bf Proof#1:} }{\hspace*{\fill}$\square$\medskip\par}
\newcommand{\ind}{1\hspace{-0.098cm}\mathrm{l}}
\def\ed#1{ {\mathbf 1}_{ \{#1  \}}}             

\def\CC{{\mathcal C}}

\def\FF{{\mathcal F}}

\newcommand{\eps}{\varepsilon}

\def\UU{{\mathcal{U}}}

\def\bark{\overline{k}}

\def\di{{\mathfrak{D}}}
\def\txi{{\widetilde{\xi}}}
\def\tR{{\widetilde{R}}}
\def\Vol{\textrm{vol}_d}
\def\Volone{\textrm{vol}_1}

\def\dalpha{\bar\alpha}
\def\KKK{\mathcal{K}}
\let\BFseries\bfseries\def\bfseries{\BFseries\mathversion{bold}} 

\title{How complex is a random picture?}
\author{Frank Aurzada and Mikhail Lifshits}

\begin{document}

\maketitle


\begin{abstract}
  We study the amount of information that is contained in ``random pictures'', by which
  we mean the sample sets of a Boolean model. To quantify the notion ``amount of information'', two closely connected questions are investigated: on the one hand, we study the probability  that a large number of balls is needed for a full reconstruction of a Boolean model sample
  set. On the other hand, we study the quantization error of the Boolean model w.r.t.\ the
  Hausdorff distance as a distortion measure.
\end{abstract}




\section{Introduction and results}
\subsection{Introduction}
We are interested in quantifying the amount of information contained in certain random
pictures. Let us first fix some notation. We work
in dimension $d\geq 1$. Let $(\xi_i)_{i\geq 1}$ be i.i.d.\ random variables uniformly
distributed in $[0,1]^d$, $(R_i)_{i\geq 1}$ be i.i.d.\ positive random varibles, and
let $N$ be a Poisson random variable with parameter $\lambda$. Assume that $(\xi_i)$, $(R_i)$,
and $N$ are independent.

Define a ``random picture'' by
$$
   S:= \bigcup_{i=1}^N B(\xi_i,R_i)\cap [0,1]^d.
$$
Here, $B(x,r)$ is a ball with centre $x$ and radius $r$, where for the time being any norm
$||.||$ on $\R^d$ is fine. The set $S$ is a union of balls in $[0,1]^d$; imagine the
balls are painted `black', thus we have a black picture over white background.

We are interested in the (lossy) encoding of the picture $S$ by a finite number of bits. This
problem, the well-known {\it quantization problem}, will be described below. It turns out
that for the analysis of the quantization problem, the following random variable is crucial.
Moreover, we believe that it may be of independent interest. Define the
{\it effective number of balls visible in the picture} as
$$
  K:= \min\{ r\geq 1 | \exists \text{$i_1,\ldots, i_r \in \{1,\ldots,N\}$} :
  S = \bigcup_{s=1}^r B(\xi_{i_s},R_{i_s})\cap[0,1]^d \}.
$$
In other words, with $K$ balls one can reproduce the black picture $S$ exactly as with the
original $N$ balls.

We are interested in the upper tail of $K$, i.e.\ $\P[K\geq n]$ when $n\to\infty$. This means
we study the probability that one needs many balls in order to reconstruct the picture $S$.
In particular, we would like to understand when one can ``save balls'' w.r.t.\ the original
Poisson number of balls $N$. To make this more precise, note that clearly $K\leq N$, and so
\[
   \P[K\geq n]\leq \P[N\geq n] = \exp( -n\log n \cdot (1+o(1))),\qquad n\to\infty.
\]
We would like to show that the upper tail of $K$ is thinner, i.e.\ for some $a>1$
$$
   \P[ K\geq n] = \exp\left( - a \cdot n\log n \cdot  (1+o(1))\right),\qquad n\to\infty.
$$
It turns out that this question is non-trivial and interesting. The answer depends on
the dimension $d$, on the type of norm used, and on the distribution of the radii
$\mathcal{L}(R_1)$.

Boolean models are fundamental objects in stochastic geometry and have a large range of applications, \cite{sg1,sg2}. However, to the knowledge of the authors, until recently mostly the {\it average of observables} of Boolean models are studied. Often this plays a role when estimating parameters of the model in applications.
On the contrary, the present paper deals with {\it rare events}, i.e.\  with large
deviation probabilities.

As mentioned above, the upper tail of the random variable $K$ is an essential ingredient
for solving the so-called quantization problem, which we recall now. Let an arbitrary
norm $||.||$ be fixed on $\R^d$. Let $d_H$ denote the corresponding Hausdorff distance
between the closed subsets of $\R^d$. We define the respective quantization error for
pictures by
$$
    D^{(q)}(r):=\inf_{\#\CC\leq e^r} \E \min_{A\in\CC} d_H(S,A),\qquad r>0.
$$

Here, the sets $\CC$ are called codebooks and the upper index $(q)$ stands for
``quantization''. The idea is that the ``analog'' signal $S$
should be encoded by an element $A\in\CC$. This incurs an error, $d_H(A,S)$, measured in
Hausdorff distance. Losely speaking, $D^{(q)}$ is then the minimal average error over all
codebooks $\CC$ of a size not exceeding $e^r$. We are interested in letting $r\to\infty$,
that is, the size of the codebooks grows; and we would like to understand the rate
of decay of the corresponding quantization error.

Basic references for quantization problems are \cite{lp,ct,kolmogorov}. The analysis
of the quantization started in the 40ies of the 20th century, and mainly finite-dimensional
quantization was the subject of interest until about 2000. Since then, research has shifted
to infinite dimensional quantization, e.g.\ for Brownian motion, fractional Brownian motion,
L\'evy processes, etc.\ attaining values in some function spaces, therefore called functional quantiation, see e.g.\ \cite{dereichscheutzow,creutzigetal,der1,der2,lg1,muellergronbachritter,c1,glp1,altmayeretal} and references therein for a selection, and their applications to numerical probability, see e.g.\ \cite{pagessurvey}. In the present paper, the
signal attains values in a more abstract space, namely, in the class of all compact subsets of $[0,1]^d$.

Results similar to those in this paper are obtained in \cite{vormoor,vormoorphd}. In \cite{vormoor}, certain types of jump processes are studied that resemble (and contain a as special case) compound Poisson processes with values in an abstract space. In the PhD thesis \cite{vormoorphd}, the question we are interested in in this paper appears for the first time.  We comment on the particular relation to the present results below.

The rest of this paper is structured as follows. In Section~\ref{sec:resultstail}, we summarize
our results for the large deviation probabilities of $K$. In Section~\ref{sec:resultsquantization},
the results for the quantization error are listed. The proofs are given
in Section~\ref{sec:proofsldlower} (lower bounds for the large deviation results),
Section~\ref{sec:proofsldupper} (upper bounds for the large deviation results),
Section~\ref{sec:dimension1} (dimension $d=1$),
and Section~\ref{sec:quantization} (quantization results), respectively.

\subsection{Results for the large deviations of $K$} \label{sec:resultstail}
\paragraph*{Dimension $d=1$.} Let us start with the case $d=1$, which is particularly easy.

In $\R^1$ there is essentially one norm, thus we will work with absolute values.
The balls here are just intervals, $B(x,r)=[x-r,x+r]$. The proof of the following result is given in Section~\ref{sec:quantization}.

\begin{thm}  \label{thm:d1} Let $d=1$.
Assume that the distribution of $R_1$ has a probability density $p$ with
$p(z) \approx z^{\alpha-1}$ for $z\to 0$ and some $\alpha>0$. Then
\be \label{eqn:d1}
    \P[ K \geq n ] = \exp( - (1+\alpha) n \log n \cdot (1+o(1)) ),
    \qquad \text{as $n\to\infty$.}
\ee
\end{thm}

Here and below, the notion $p(z)\approx q(z)$ ($z\to 0$)  stands for the fact that $p(z)/q(z)$ is bounded away from zero and infinity for $z$ small enough. Likewise, we use $p(z)\sim q(z)$ if $\lim p(z)/q(z) = 1$.

For the case of constant radius, the large deviations turn out to be trivial, which is
quite natural. Namely, if $R_1\equiv c<1$ is constant, Remark~\ref{rem:empty} below shows
that $\P[ K \geq n ] = 0$, for $n>2/c$.

From now on, we assume $d\geq 2$.

\paragraph*{Constant radius.} Let us now deal with the seemingly simple case of constant radii. It turns out that the rates (and the proofs) are non-trivial and may possibly depend on the geometry of the balls.

\begin{thm}  \label{thm:constantgeneral} Assume $R_1\equiv c<1$ is constant. Then
\begin{enumerate}
\item for $\ell_1$-balls, we have
$$
\P[K\geq n] = \exp\left( - (1+\frac{1}{d-1}) n \log n \cdot (1+o(1)) \right),
$$
\item for $\ell_2$-balls, we have
\begin{eqnarray*}
&& \exp\left( - (1+\frac{2}{d-1}) n \log n \cdot (1+o(1)) \right)
\\
&\leq& \P[K\geq n] \leq \exp\left( - (1+\frac{1}{d-1}) n \log n \cdot (1+o(1)) \right),
\end{eqnarray*}
\item for $\ell_\infty$-balls, we have
$$
\exp\left( - (1+\frac{1}{d-1}) n \log n \cdot (1+o(1)) \right) \leq \P[K\geq n].
$$
\end{enumerate}
\end{thm}

These results are proved in Section~\ref{sec:constantlower} (lower bounds) and
Section~\ref{sec:constantupper} (upper bounds).

\paragraph*{Radius distribution with density.} Now we deal with a radius distribution that
has a probability density. The result here is a summary of the results
in Sections~\ref{sec:gen1}, \ref{sec:gen2} (lower bounds) and \ref{sec:gen3} (upper bounds),
where slightly more general results are stated and proved.

\begin{thm}  \label{thm:genericgeneral}
Assume that the radius distribution has a probability density $p$
with $p(z)\approx z^{\alpha-1}$, for $z\to 0$. Set $\dalpha:=\alpha\wedge 1$.
Furthermore, assume that $p$ is bounded. Then for any norm on $\R^d$
\begin{eqnarray*}
&& \exp\left( - \left(1+\frac{\alpha}{d}\right) n \log n \cdot (1+o(1)) \right)
\\
&\leq& \P[K\geq n] \leq \exp\left( - \left(1+\frac{\dalpha}{d}\right) n \log n \cdot (1+o(1)) \right).
\end{eqnarray*}

Additionally, for $\ell_1$-balls,
$$
\exp\left( - (1+\frac{1}{d-1}) n \log n \cdot (1+o(1)) \right) \leq \P[K\geq n]
$$
and for $\ell_2$-balls,
$$
\exp\left( - (1+\frac{2}{d-1}) n \log n \cdot (1+o(1)) \right) \leq \P[K\geq n].
$$
\end{thm}

%
%

\subsection{Results for the quantization error}\label{sec:resultsquantization}

\paragraph*{Dimension $d=1$.} Also here we start with dimension $d=1$. First, we treat the case of constant radius.

\begin{thm} \label{thm:d1quantization}
Let $d=1$. Assume that the radius is constant $R_1\equiv c<1/2$. Then for constants $c_1,c_2>0$ and large enough $r$ we have
\begin{equation}
c_1 e^{-r/2m}\le    D^{(q)}(r) \le c_2 e^{-r/2m}, \label{d1c}
\end{equation}
where $m:=\max\{k\in\N : 2kc<1\}$. For $c\geq 1/2$, the statement holds for $m=1/2$.
\end{thm}

Now we look at $d=1$ and non-constant radius. Note that Lemma~\ref{lem:quantuppergeneral} below immediately turns an upper bound for the large deviations of $K$ into an upper bound for the quantization rate. In particular, from Theorem~\ref{thm:d1} we obtain the following.

\begin{cor} Let $d=1$. Assume that the distribution of $R_1$ has a probability density $p$ with $p(z) \leq c z^{\alpha-1}$ for $z\to 0$ and some $\alpha>0$. Then
$$
    D^{(q)}(r) \leq \exp\left(-\sqrt{(1+\alpha) r \log r} \cdot (1+o(1))\right)\qquad \text{as $r\to\infty$.}
$$
\end{cor}

Lower bounds could be obtained in a similar matter as for larger dimensions below, we do not pursue this here to keep the exposition comprehensive. We further mention that Theorem 4.2.1 in \cite{vormoorphd} treats the case $\alpha=1$, which is extended to general $\alpha$ here.

For the rest of this section, we deal with $d\geq 2$.

\paragraph*{Constant radius.} Let us first consider the case of constant radius.

\begin{thm} \label{thm:coding-2-3} For $\ell_1$-balls and constant radius
we have
$$
    D^{(q)}(r)
    = \exp\left( - \sqrt{ \frac{2}{d-1}\,r \log r}  \cdot (1+o(1))\right),
    \qquad \text{as $r\to\infty$}.
$$
 For $\ell_2$-balls and constant radius
we have
\begin{multline*}
      \exp\left( - \sqrt{ \frac{4(d+1)}{d(d-1)}\,r \log r}  \cdot (1+o(1))\right)
\\
    \leq D^{(q)}(r)\leq
    \exp\left( - \sqrt{ \frac{2}{d-1}\,r \log r}  \cdot (1+o(1))\right) ,
    \qquad \text{as $r\to\infty$}.
\end{multline*}
\end{thm}

\paragraph*{Radius distribution with density.}  Finally, we deal with a radius distribution with a density.

\begin{thm} \label{thm:coding1}
Assume that the distribution of $R_1$ has a probability density
$p$ with $p(z)\approx z^{\alpha-1}$ for $z\to 0$ for some $\alpha>0$.
Set $\dalpha:=\alpha\wedge 1$. Then
$$
   \exp\left( - b \, \sqrt{r \log r}  \cdot (1+o(1))\right)
   \leq D^{(q)}(r)
   \leq \exp\left( -\, \overline{b} \sqrt{r \log r}  \cdot (1+o(1))\right),
$$
as $r\to\infty$, where $b:=\sqrt{ \frac{2 (1+\alpha/d)}{d+1}}$ and $\overline{b}:=\sqrt{ \frac{2 (1+\dalpha/d)}{d+1}}$.
\end{thm}

Theorem~4.3.2 and Theorem~4.3.3 in \cite{vormoorphd} treat the case $\alpha=1$ and give the right lower bound in that case. Here, we show that this lower bound is sharp and extend both bounds to more general $\alpha$.

\section{Lower bounds for the large deviation results} \label{sec:proofsldlower}

\subsection{Constant radius for $\ell_1$-, $\ell_2$-, $\ell_\infty$-norms} \label{sec:constantlower}

\subsubsection{ $\ell_2$-norm}
\begin{prop} \label{prop:lowerc2}
Assume that the radius is a.s.\ constant $R_1\equiv c<1$ and we work with $\ell_2$-balls.
 If $d\geq2$, then
$$
    \P[ K \geq n ] \geq \exp( - (1+\frac{2}{d-1}) n \log n (1+o(1)) ),
    \qquad \text{as $n\to\infty$.}
$$
\end{prop}

\begin{proof}
Consider the following collection of boxes:
$$
   \left\{
   \prod_{m=1}^{d-1} \left[ \frac{k_m}{(2n)^{1/(d-1)}}+\frac{1/4}{(2n)^{1/(d-1)}},
                        \frac{k_m}{(2n)^{1/(d-1)}}+\frac{3/4}{(2n)^{1/(d-1)}}\right]
   \right\}
  \times \left[0,\frac{c_1}{n^{2/(d-1)}}\right]
$$
with $k_m\in\{0,\ldots,\lfloor (2n)^{1/(d-1)}\rfloor-1\}$ and $c_1:=2^{-(4+2/(d-1))}$.
The number of boxes being of order $2n$, we may
choose among them $n$ distinct boxes, say $V_1, \ldots, V_n$. Define the following event:
$$
    E:=E_n:=\{N=n\}\cap \bigcup_{\text{$\pi$ permutation of  $\{1,\ldots,n\}$}}
      \{\xi_i\in V_{\pi(i)}, i=1,\ldots,n\}.
$$
We will show that -- given $E$ -- each ball $B(\xi_i,R_1)$ ($i=1,\ldots,n$) contains
a point that is not covered by any other ball
$B(\xi_j,R_1)$, $j=1,\ldots,n$, $j\neq i$.
Therefore, $E$ implies $K\geq n$. More precisely, the point $x_i:=\xi_i+(0,\ldots,0,R_1)$
is obviously in the ball $B(\xi_i,R_1)$ and it is not covered by any other ball:
Indeed, for $j\neq i$:
\begin{eqnarray*}
     ||x_i-\xi_j||_2^2 &=& |\xi_i^{(d)}+R_1-\xi_j^{{(d)}}|^2
     + \sum_{m=1}^{d-1} |\xi_i^{(m)}-\xi_j^{(m)}|^2
\\
     &=& |\xi_i^{{(d)}}-\xi_j^{{(d)}}|^2+R_1^2-2R_1 |\xi_i^{{(d)}}-\xi_j^{{(d)}}|^2
     + \sum_{m=1}^{d-1} |\xi_i^{(m)}-\xi_j^{(m)}|^2
\\
     &\geq & 0+R_1^2-2R_1 c_1 n^{-2/(d-1)}+\left( \frac{1/2}{(2n)^{1/(d-1)}}\right)^2
>R_1^2,
\end{eqnarray*}
by the choice of $c_1$. Further, note that for large enough $n$ the point $x_i$ is indeed
in $[0,1]^d$, as $R_1\equiv c<1$.

Therefore, the event $E$ implies $K\geq n$ and so
\begin{multline*}
   \P[K\geq n]\geq \P[E]
    = \frac{\lambda^n}{n!} e^{-\lambda} \cdot n!
    \cdot\left(
    \left( \frac{1/2}{(2n)^{1/(d-1)}} \right)^{d-1}  \cdot
    c_1 n^{-2/(d-1)}
    \right)^n
\\
    \geq \exp\left( - \left(1+\frac{2}{d-1}\right) n \log n (1+o(1)) \right).
\end{multline*}
\end{proof}

\subsubsection{ $\ell_1$-norm}
Now we consider the case of constant radius and $\ell_1$-balls.

\begin{prop} \label{prop:lowerc1}
Assume that the radius is a.s.\ constant $R_1=c<1$ and we work with $\ell_1$-balls.
If $d\geq2$, then
$$
    \P[ K \geq n ] \geq \exp\left( - \left(1+\frac{1}{d-1}\right) n \log n (1+o(1)) \right),
    \qquad \text{as $n\to\infty$.}
$$
\end{prop}

The proof is completely analogous to the $\ell_2$-norm case, with the only difference being
the possibility to keep the first component in a larger set due to the geometric structure
of $\ell_1$-balls. This results in the larger bound.

\begin{proof}
Consider the following collection of boxes:
$$
     \left\{
     \prod_{m=1}^{d-1} \left[ \frac{k_m}{(2n)^{1/(d-1)}}+\frac{1/4}{(2n)^{1/(d-1)}},
           \frac{k_m}{(2n)^{1/(d-1)}}+\frac{3/4}{(2n)^{1/(d-1)}}\right]
      \right\}
       \times \left[0,\frac{c_2}{n^{1/(d-1)}}\right] ,
$$
with $k_m\in\{0,\ldots,\lfloor (2n)^{1/(d-1)}\rfloor-1\}$.
Here, $c_2:=2^{-(2+1/(d-1))}$.
The number of boxes being of order $2n$, we may choose among them $n$ distinct boxes,
say $V_1, \ldots, V_n$. Define the following event:
$$
   E:=E_n:=\{N=n\}\cap\bigcup_{\text{$\pi$ permutation of $\{1,\ldots,n\}$ }}\{\xi_i\in V_{\pi(i)},
   i=1,\ldots,n\}.
$$
We will show that -- given $E$ -- each ball $B(\xi_i,R_1)$ ($i=1,\ldots,n$) contains а point
that is not covered by any other ball $B(\xi_j,R_1)$, $j=1,\ldots,n$,
$j\neq i$. Therefore, $E$ implies $K\geq n$. More precisely, the point $x_i:=\xi_i+(0,\ldots,0,R_1)$
is obviously in the ball $B(\xi_i,R_1)$ and it is not covered by any other ball:
Indeed, for $j\neq i$:
\begin{eqnarray*}
    ||x_i-\xi_j||_1 &=& |\xi_i^{{(d)}}+R_1-\xi_j^{{(d)}}|
    + \sum_{m=1}^{d-1} |\xi_i^{(m)}-\xi_j^{(m)}|
\\
   &\geq & R_1 - |\xi_i^{{(d)}}-\xi_j^{{(d)}}| + \sum_{m=1}^{d-1} |\xi_i^{(m)}-\xi_j^{(m)}|
\\
   &\geq & R_1 - c_2 n^{-1/(d-1)} + \frac{1/2}{(2n)^{1/(d-1)}}
   > R_1,
\end{eqnarray*}
by the choice of $c_2$. Further, note that for large enough $n$ the point $x_i$ is indeed
in $[0,1]^d$, as $R_1=c<1$.

Therefore, the event $E$ implies $K\geq n$. Thus,
\begin{multline}\label{eqn:againsameprob}
    \P[K\geq n]\geq \P[E]
    =\frac{\lambda^n}{n!} e^{-\lambda} \cdot n!
    \cdot\left(
       \left( \frac{1/2}{(2n)^{1/(d-1)}} \right)^{d-1}
         \cdot  c_2 n^{-1/(d-1)}
    \right)^n
\\
   \geq \exp\left( - \left(1+\frac{1}{d-1}\right) n \log n (1+o(1)) \right).
\end{multline}
\end{proof}

\subsubsection{ $\ell_\infty$-norm}
Now we consider the case of constant radius and $\ell_\infty$-balls.
The approach is very similar to the previous ones but the centers of the
balls are placed near a ``diagonal'' hyperplane.

\begin{prop}
Assume that the radius is a.s.\ constant $R_1=c<1$ and we work with $\ell_\infty$-balls.
If $d\geq2$, then
$$
    \P[ K \geq n ] \geq \exp\left( - \left(1+\frac{1}{d-1}\right) n \log n (1+o(1)) \right),
    \qquad \text{as $n\to\infty$.}
$$
\end{prop}

\begin{proof}
Let us fix $\rho_1,\rho_2$ such that $R_1<\rho_1<\rho_2<1$ and consider the $(d-1)$-dimensional nonempty
set
\[
   H:=\left\{ x\in[0,1]^d:\, \sum_{m=1}^d x^{(m)} = n\rho_2, \min_{1\leq m\leq d} x^{(m)}>\rho_1 \right\}.
\]
For sufficiently small $c_1=c_1(d,\rho_1,\rho_2,c)$  we may choose $n$ points $\beta_1,...,\beta_n$
in $H$ such that $||\beta_i-\beta_j||_1 > c_1 n^{-1/(d-1)}$ for all $i\neq j$. (We stress that we take $\ell_1$-norm here.)

Consider the following collection of boxes:
\[
    V_i:= B(\beta_i,c_2 n^{-1/(d-1)}),  1\leq i \leq n,
\]
with $c_2<  c_1/(4d)$.

Define the following event:
$$
   E:=E_n:=\{N=n\}\cap\bigcup_{\text{$\pi$ permutation of $\{1,\ldots,n\}$ }}\{\xi_i\in V_{\pi(i)},
   i=1,\ldots,n\}.
$$
We will show that -- given $E$ -- each ball $B(\xi_i,R_1)$ ($i=1,\ldots,n$) contains
a point that is not covered by any other ball $B(\xi_j,R_1)$,
$j\neq i$.

Consider the point $x_i:=\xi_i-(R_1,\ldots,R_1)$. For $1 \leq m \leq d$ we have
\begin{eqnarray*}
   x_i^{(m)} &=& \xi_i^{(m)} - R_1
   \geq\beta_{\pi(i)}^{(m)} - ||\xi_i-\beta_{\pi(i)}||_\infty - R_1
\\
   &\geq& \rho_1 - c_2 n^{-1/(d-1)} - R_1 >0
\end{eqnarray*}
for sufficiently large $n$, which yields $x_i\in [0,1]^d$.
It is also obvious that $x_i \in B(\xi_i,R_1)$.

We show now that $x_i$ is not covered by any other ball.
Let $x'_i:=\beta_{\pi(i)}-(R_1,\ldots,R_1)$. Then
\[
  ||x_i-x'_i||_\infty = ||\xi_i-\beta_{\pi(i)}||_\infty  \leq c_2 n^{-1/(d-1)}.
\]
It follows that for any $j\neq i$ we have
\begin{eqnarray} \notag
    ||x_i-\xi_j||_\infty
     &\geq& ||x'_i-\beta_{\pi(j)}||_\infty - ||x_i-x'_i||_\infty
     -||\xi_j-\beta_{\pi(j)}||_\infty
\\     \label{eqn:inftyconstR1}
    &\geq& ||x'_i-\beta_{\pi(j)}||_\infty -  2c_2 n^{-1/(d-1)}.
\end{eqnarray}
It remains to evaluate $ ||x'_i-\beta_{\pi(j)}||_\infty$.
Since $\beta_{\pi(i)},\beta_{\pi(j)}\in H$, we have
\[
  \sum_{m=1}^{d} \left( \beta_{\pi(j)}^{(m)} - \beta_{\pi(i)}^{(m)}\right)
  =  \sum_{m=1}^{d} \beta_{\pi(j)}^{(m)} -   \sum_{m=1}^{d} \beta_{\pi(i)}^{(m)} =0.
\]
In other words,
\[
    \sum_{m=1}^{d} \left( \beta_{\pi(j)}^{(m)} - \beta_{\pi(i)}^{(m)}\right)_+
    =  \sum_{m=1}^{d} \left( \beta_{\pi(j)}^{(m)} - \beta_{\pi(i)}^{(m)}\right)_-.
\]
On the other hand, by construction,
$$\sum_{m=1}^{d} \left( \beta_{\pi(j)}^{(m)} - \beta_{\pi(i)}^{(m)}\right)_+
    +  \sum_{m=1}^{d} \left( \beta_{\pi(j)}^{(m)} - \beta_{\pi(i)}^{(m)}\right)_-
    =||\beta_{\pi(j)} - \beta_{\pi(i)}||_1  > c_1 n^{-1/(d-1)}.
$$
It follows that
\[
   \sum_{m=1}^{d} \left( \beta_{\pi(j)}^{(m)} - \beta_{\pi(i)}^{(m)}\right)_+
   > c_1 n^{-1/(d-1)}/2
\]
and
\begin{eqnarray} \notag
   ||x'_i-\beta_{\pi(j)}||_\infty &\geq&  \max_{1\leq m \leq d}
   \left( \beta_{\pi(j)}^{(m)} - {x'_i}^{(m)} \right)
\\  \notag
   &=& \max_{1\leq m \leq d} \left( \beta_{\pi(j)}^{(m)} - \beta_{\pi(i)}^{(m)}\right) +R_1
\\  \notag
     &\geq& \frac{1}{d} \sum_{m=1}^{d}
      \left( \beta_{\pi(j)}^{(m)} - \beta_{\pi(i)}^{(m)}\right)_+ + R_1
\\  \label{eqn:inftyconstR2}
    &>& c_1 n^{-1/(d-1)}/(2d) + R_1.
\end{eqnarray}
By using the bounds \eqref{eqn:inftyconstR1}, \eqref{eqn:inftyconstR2}, and
the definition of $c_2$, we obtain
\[
    ||x_i-\xi_j||_\infty \geq  c_1 n^{-1/(d-1)}/(2d) + R_1 -  2c_2 n^{-1/(d-1)}> R_1.
\]
This means $x_i\not\in B(\xi_i,R_1)$, as claimed.
Therefore, the event $E$ implies $K\geq n$ and so
\begin{multline*}
   \P[K\geq n]\geq \P[E]
    = \frac{\lambda^n}{n!} e^{-\lambda} \cdot n! \cdot\left( c_2 n^{-1/(d-1)} \right)^{dn}
\\
    \geq \exp\left( - \left(1+\frac{1}{d-1}\right) n \log n (1+o(1)) \right).
\end{multline*}
\end{proof}

\subsection{Generic radius: Lower bound via small balls} \label{sec:gen1}

 The following result is valid for arbitrary norm in $\R^d$, \, $d\geq1$.

\begin{prop} \label{prop:lowerSB}
Assume that the distribution of $R_1$ has a probability density $p$ with
$p(z)\geq c z^{\alpha-1}$ for small $z$ and some constants $c>0$ and $\alpha>0$. Then
$$
    \P[ K \geq n ] \geq \exp( - (1+\alpha/d) n \log n (1+o(1)) ),
    \qquad \text{as $n\to\infty$.}
$$
\end{prop}

\begin{proof} Let $n\geq (2^{1/d}-1)^d$. A lower bound is obtained from the following
scenario. Consider the following collection of cubic boxes:
$$
    \prod_{m=1}^d \left[ \frac{k_m}{(2n)^{1/d}}+\frac{1/4}{(2n)^{1/d}},
          \frac{k_m}{(2n)^{1/d}}+\frac{3/4}{(2n)^{1/d}}\right],
          \qquad k_m\in\{0,\ldots,\lfloor (2n)^{1/d}\rfloor-1\}.
$$
The number of boxes being of order $2n$, we may choose among them $n$ distinct boxes,
say $V_1, \ldots, V_n$. Define the following event:
$$
   E:=E_n:=\{N=n\}\cap\bigcup_{\text{$\pi$ permutation of $\{1,\ldots,n\}$ }} E_\pi,
$$
where
$$
   E_\pi := \{ \xi_i\in V_{\pi(i)}, R_i \in [c_1 n^{-1/d}, c_2 n^{-1/d}],
   \forall i=1,\ldots,n \},
$$
with some constants $c_2>c_1>0$.  The constant $c_2$ depending on the norm under
consideration can be chosen so small that
for distinct $i$ and $j$ the balls $B(\xi_i,R_i)$ and $B(\xi_j,R_j)$ are disjoint.
Therefore, the event $E$ implies $K\geq n$.

Finally, note that
\begin{eqnarray*}
   \P[K\geq n]&\geq& \P[E]
\\
   &=&  n! \cdot \frac{\lambda^n}{n!} e^{-\lambda} \cdot
     \left( \left( \frac{1/2}{(2n)^{1/d}} \right)^{d} \cdot
     \int_{c_1 n^{-1/d}}^{c_2 n^{-1/d}} p(z) \dd z \right)^n
\\
   &\geq &  \left(\lambda 2^{-(d+1)} \right)^n
   e^{-\lambda}\cdot n^{-n}
   \cdot \left( \int_{c_1 n^{-1/d}}^{c_2 n^{-1/d}} c z^{\alpha-1} \dd z\right)^n
\\
   &\geq &   \left(\lambda 2^{-(d+1)}(c_2^\alpha-c_1^\alpha)/\alpha \right)^n
    e^{-\lambda} \cdot n^{-n} \cdot n^{-\alpha n/d}
\\
   & =& \exp( - (1+\alpha/d) n \log n \cdot(1+o(1))).
\end{eqnarray*}
\end{proof}

\subsection{Generic radius: Lower bound via surfaces}  \label{sec:gen2}
\begin{prop} Assume the radius distribution has a bounded density. Then,
\begin{itemize}
\item for $\ell_1$-norm balls,
$$\P[ K\geq n] \geq \exp\left( - \left(1+\frac{1}{d-1}\right) n \log n \cdot (1+o(1))\right);$$
\item for $\ell_2$-norm balls,
$$\P[ K\geq n] \geq \exp\left( - \left(1+\frac{2}{d-1}\right) n \log n \cdot (1+o(1))\right).$$
\end{itemize}
\label{prop:LBsurface}
\end{prop}

We will prepare the proof with the following lemma. It shows that any probability
distribution with bounded density
has many intervals in its support with the mass proportional to the length of those intervals
or larger.

\begin{lem} \label{lem:concentrationfctn2}
Let $R\in[0,1]$ be a random variable having a bounded probability density $p$, say
\begin{equation} \label{eqn:densitybound}
\esssup_{x\in[0,1]} p(x) \leq c < \infty.
\end{equation}
Then for any $\beta\in[0,1]$ and any $\delta>0$ we have
$$
   \Volone
   \lbrace x\in [0,1] : \P[ R\in [x,x+\beta] ] > \delta \beta \rbrace
   > \frac{1 - 2 c \beta - \delta}{c}.
$$
\end{lem}

%

\begin{proof} Set, for ease of notation, $\P[ R\in [x,x+\beta] ]=:Q_x(\beta)$.
First note that we have
\begin{eqnarray*}
    \int_0^{1-\beta} Q_x(\beta) \dd x
    &=& \int_0^{1-\beta} \int_x^{x+\beta} p(y) \dd y  \dd x
    = \int_0^1 \int_{0\vee (y-\beta)}^{(1-\beta)\wedge y}   \dd x p(y) \dd y
\\
    &\geq&  \int_\beta^{1-\beta} \int_{y-\beta}^y  \dd x p(y) \dd y
   = \beta \int_\beta^{1-\beta} p(y) \dd y
\\
 &\geq&  \beta \left( \int_0^1 p(y) \dd y - 2 c \beta \right)
   = \beta (1- 2c\beta).
\end{eqnarray*}

On the other hand,  note that $Q_x(\beta)\leq c \beta$. Therefore,
\begin{eqnarray*}
    \int_0^{1-\beta} Q_x(\beta) \dd x
    &=& \int_0^{1-\beta} Q_x(\beta)\ind_{Q_x(\beta)\leq \delta \beta}\dd x
        +  \int_0^{1-\beta} Q_x(\beta)\ind_{Q_x(\beta)> \delta \beta}\dd x
\\
    &\leq & \delta \beta +  c\beta \int_0^1 \ind_{Q_x(\beta)> \delta \beta}\dd x.
\end{eqnarray*}
It follows that
\[
    1-2c\, \beta \leq \delta + c\ \Volone
      \lbrace x\in [0,1] : \P[ R\in [x,x+\beta] ] > \delta \beta \rbrace.
\]
Rearranging the terms gives the claim.
\end{proof}

\begin{proof}[ of Proposition~\ref{prop:LBsurface}]
We shall proceed in a number of steps: after some preparations,
we define a scenario, estimate its probability, and then show that
the scenario implies $K\geq n$.

{\it Preparation: $(d-1)$-dimensional boxes.} Set $\eps_n:=(2n)^{-1/(d-1)}/2$.
Let us consider the following collection of boxes:
$$
    \prod_{m=1}^{d-1} \big[ 2 k_m \eps_n , (2k_m+1)\eps_n\big],\qquad
    k_m \in\{ 0,\ldots, \lfloor \frac{\eps_n^{-1}-1}{2}\rfloor\}.
$$
The number of such boxes is greater or equal to
$$
   \Big( \lfloor \frac{\eps_n^{-1}-1}{2}\rfloor + 1\Big)^{d-1}
   \sim (2\eps_n)^{-(d-1)} = 2n.
$$
Therefore, one can choose $n$ of these boxes, say $V_1,\ldots, V_n$.
The main feature of these boxes is that
\begin{equation} \label{eqn:differentboxes}
     \forall i\neq j ~ \forall y\in V_i, y'\in V_j ~:~
     \max_{1\leq m \leq d-1} |y^{(m)} - (y')^{(m)}| \geq \eps_n.
\end{equation}

{\it Preparation: support of the $d$-th component.}  Let $c$ be the bound
of the density (as in (\ref{eqn:densitybound})). It follows immediately from
Lemma~\ref{lem:concentrationfctn2} (with $\delta=1/4$) that
for any $\beta\in[0,\frac{1}{4c}]$ we have
$$
   \Volone\left\lbrace x\in [0,1] : \P[ R\in [x,x+\beta] ] >  \frac{\beta}{4}
   \right\rbrace
   > \frac{1}{4c}.
$$
This implies that for $\beta<1/(8c)$ we have
\begin{eqnarray*}
     \frac{1}{8c}
     &<& \Volone \left\lbrace x\in [0,1-\beta] : \P[ R\in [x,x+\beta] ]
        > \frac{\beta}{4} \right\rbrace
\\
     &=& \Volone \left\lbrace z\in [0,1-\beta] : \P[ R+z\in [1-\beta,1] ]
        > \frac{\beta}{4} \right\rbrace.
\end{eqnarray*}
Let us denote
$$
   Z(\beta):=\left\lbrace z\in [0,1-\beta] : \P[ R+z\in [1-\beta,1] ]
      > \frac{\beta}{4} \right\rbrace.
$$
We shall use this set as the support of the $d$-th component in the scenario we
will construct.
It will be used for $\beta=\beta_n\to 0$ so that the assumption $\beta<1/(8c)$
is satisfied for $n$ large enough.

{\it Definition of the scenario and evaluation of its probability.}
Define the `tubes'
$$
W_i:=V_i \times Z(\beta_n), \qquad i=1,\ldots, n,
$$
where $\beta_n$ is chosen later according to the involved norm.

Consider the following scenario:
$$
   E:=E_n:=\{N=n\}\cap
   \bigcup_{\text{$\pi$ permutation of $\{1,\ldots,n\}$ }} E_\pi,
$$
where
$$
   E_\pi := \{ \xi_i\in W_{\pi(i)}, \xi_i^{(d)}+R_i \in [1-\beta_n,1],
   \forall i=1,\ldots,n\}.
$$
By using the projection
$\sigma(x^{(1)},\ldots,x^{(d)}):=(x^{(1)},\ldots,x^{(d-1)})$,
we can estimate the probability of $E$ as follows:
\begin{eqnarray}
    \P[E] &\geq & \frac{\lambda^n}{n!} \, e^{-\lambda} \cdot n! \cdot
    \left( \P[ \sigma(\xi_1)\in V_1]\cdot
    \P[\xi_1^{(d)}\in Z(\beta_n), \xi_1^{(d)}+R_1 \in [1-\beta_n,1]] \right)^n
    \notag
\\
    &\geq & \lambda^n e^{-\lambda} \left( \eps_n^{d-1} \cdot
    \P[\xi_1^{(d)}\in Z(\beta_n)] \,
    \inf_{z\in Z(\beta_n)}
    \P[z+R_1 \in [1-\beta_n,1]] \right)^n
    \notag
\\
    &\geq & \lambda^n e^{-\lambda}
    \left( \eps_n^{d-1} \cdot \frac{1}{8c}\, \frac{\beta_n}{4}\right)^n
    \notag
\\
    &=& \left(\frac{\lambda}{2^{d+5} c}\right)^n e^{-\lambda}
    n^{-n} \beta_n^n,
    \label{eqn:rate}
\end{eqnarray}
where we used that the $\xi_i$ are uniformly distributed in $[0,1]^d$ in the second
and third step. Later, we will chose $\beta_n$ (polynomially decaying in $n$)
according to the involved norm.

{\it Scenario $E$ implies $K\geq n$.} We now proceed to showing that the
scenario $E=E_n$ implies $K\geq n$.

For this, it is sufficient to show that under $E$, for any $i$ the following auxiliary
point $x_i$ belongs to the ball $B(\xi_i,R_i)$ but it is not covered by any other ball.
Thus, none of the balls $B(\xi_i,R_i)$ can be left out when representing the picture $S$.

Define $x_i:=\xi_i + (0,\ldots,0,R_i)$. Clearly,
\[
   \| \xi_i - x_i\|_\infty =  \|\xi_i - x_i\|_2 = \| \xi_i - x_i\|_1 = R_i,
\]
thus $x_i\in B(\xi_i,R_i)$. It remains to show that
$x_i\not\in B(\xi_j,R_j)$ for any $j\neq i$, i.e.
\begin{equation} \label{eqn:auxpointout}
    ||x_i-\xi_j||_q > R_j\quad \forall j\neq i,
\end{equation}
for respective $q\in\{1;2\}$.
This will be achieved separately for the different norms and with different choices
of the sequence $(\beta_n)$.

{\it Proof of $\eqref{eqn:auxpointout}$ for $\ell_1$-norm.}
Here we choose $\beta_n:=\eps_n/2$.

Assume for the sake of contradiction that $||x_i-\xi_j||_1\leq R_j$ for some
$i\neq j$.

Since $\xi_j^{(d)} + R_j \in [1-\beta_n,1]$ on the event $E$, for any
$z\in[1-\beta_n,1]$ we have
$$
   \beta_n \geq | (\xi_j^{(d)} + R_j ) - z| = |R_j - (z-\xi_j^{(d)})|
   \geq R_j - |z-\xi_j^{(d)}|
$$
and so $|z-\xi_j^{(d)}| \geq R_j - \beta_n$.
Letting $z:=x_i^{(d)}=\xi_i^{(d)}+R_i\in[1-\beta_n,1]$, we have
\be \label{eqn:xixijd}
   |x_i^{(d)} -\xi_j^{(d)}| \geq R_j - \beta_n.
\ee
It follows that
\begin{eqnarray*}
   R_j &\geq& ||x_i -\xi_j||_1
   = \sum_{m=1}^{d-1} | x_i^{(m)}-\xi_j^{(m)}|
   + | x_i^{(d)}-\xi_j^{(d)}|
\\
   &\geq &\sum_{m=1}^{d-1} | \xi_i^{(m)}-\xi_j^{(m)}| + R_j - \beta_n
\\
   &\geq &\max_{1\leq m \leq d-1} | \xi_i^{(m)}-\xi_j^{(m)}| + R_j - \beta_n.
\end{eqnarray*}
Hence,
$$
    \max_{1\leq m \leq d-1} | \xi_i^{(m)}-\xi_j^{(m)}|
    \leq \beta_n = \eps_n/2,
$$
in contradiction to (\ref{eqn:differentboxes}).
Therefore, we must have $||x_i-\xi_j||_1> R_j$.

{\it Proof of $\eqref{eqn:auxpointout}$ for $\ell_2$-norm.}
Here we choose $\beta_n:=\eps_n^2/3$.

Assume for the sake of contradiction that $||x_i-\xi_j||_2\leq R_j$
for some $j\neq i$.

Using \eqref{eqn:xixijd} again, we have
\begin{eqnarray*}
   R_j^2 &\geq& ||x_i -\xi_j||_2^2
   = \sum_{m=1}^{d-1} | x_i^{(m)}-\xi_j^{(m)}|^2 + | x_i^{(d)}-\xi_j^{(d)}|^2
\\
   &\geq &\sum_{m=1}^{d-1} | \xi_i^{(m)}-\xi_j^{(m)}|^2 + (R_j - \beta_n)^2
\\
   &\geq &\max_{1\leq m \leq d-1} | \xi_i^{(m)}-\xi_j^{(m)}|^2
   + R_j^2 -2 R_j \beta_n + \beta_n^2.
\end{eqnarray*}
Hence,
$$
    \max_{1\leq m \leq d-1} | \xi_i^{(m)}-\xi_j^{(m)}|^2 \leq 2 R_j \beta_n
    \leq 2 \eps_n^2 / 3 < \eps_n^2,
$$
in contradiction to (\ref{eqn:differentboxes}). Therefore, we must have
$||x_i-\xi_j||_2> R_j$.

{\it Rate of $\P[E]$.} Finally, it is simple to see that the rate in \eqref{eqn:rate}
with the choices $\beta_n=\eps_n/2$ ($\ell_1$-norm) and $\beta_n=\eps_n^2/3$
($\ell_2$-norm) leads to the asserted rates in the statement of the proposition.
\end{proof}

\section{Upper bounds for the large deviation results} \label{sec:proofsldupper}
\subsection{Generic radius}  \label{sec:gen3}

The following result based on an assumption on the radius
concentration function is valid for arbitrary norm in $\R^d$.

\begin{prop} \label{prop:tailofK}
If $R_1$ is such that $\sup_{x>0} \P\big[ R_1 \in [x,x+r]\big] \leq c r^{\alpha}$
for some $c>0,\alpha\in(0,1]$ and all $r>0$,
then
$$
   \P[ K \geq n ] \leq \exp( - (1+\alpha/d) n \log n (1+o(1)) ),
   \qquad \text{as $n\to\infty$.}
$$
\end{prop}

\begin{rem}
If, for example, $R_1$ has a bounded density, then the assumption of proposition
holds with $\alpha=1$.

Note that for $\alpha>1$ one cannot expect to have a bound $\leq c r^\alpha$
in the assumption of the proposition.
\end{rem}

\begin{proof}
{\it Step 1:} Initial definitions.

In order to avoid cumbersome notations,
in this proof we assume that $n^{1/d}$ is an integer.

Let us divide the unit cube into $n$ equal boxes of side length $n^{-1/d}$.
Denote by $J_i\in\{1,\ldots,n\}$ the number of the box that contains $\xi_i$.
Further, denote by $N(k,j):=\# \{ i\leq k : J_i=j\}$ the number of balls among
the first $k$ that have their centres in the $j$-th box. If $N(k,j)>0$ we define
$$
   R^\ast(k,j):=\max\{ R_i : i\leq k, J_i=j \}
$$
the maximal radius of the balls having centres in the box $j$ among the first
$k$ balls.

{\it Step 2:} Building a collection.

We shall now gather certain balls (identified by their numbers) into a
{\it collection}.
We proceed by looking in the $k$-th step at the $k$-th ball, possibly adding it
to the collection and possibly deleting another ball from the collection. During
the whole time, the collection will maintain the following important properties:
\begin{itemize}
 \item[1)] In each step $k$, every ball (among the first $k$ balls) that is not
 included in the collection is covered by some other ball (from the first $k$
 balls).
 \item[2)] In each step $k$, for every box $j$, if $N(k,j)>0$ then a ball
 corresponding to the maximal radius with centre in that box, $R^\ast(k,j)$,
 {\it is} included in the collection.
\end{itemize}

Let us now describe the inspection procedure that will lead to a collection.
In this procedure the diameter (with respect to the norm we consider)
of the unit cube will be involved. We denote it $\di$.

At step $0$, we start with an empty collection, which certainly satisfies 1)
and 2).

When moving from $k$ to $k+1$, we first identify the box of the next ball,
$J_{k+1}$. If $N(k,J_{k+1})=0$ (i.e.\  there was no ball with the
centre located in the box $J_{k+1}$ so far),
we include the ball $k+1$ into the collection. Further, note that this step
certainly does preserve properties 1) and 2). Also note that for each box $j$ the
case $j=J_{k+1}$ and $N(k,J_{k+1})=0$  happens at most once.

If $N(k,J_{k+1})>0$, we compare $R_{k+1}$ with $R^\ast(k,J_{k+1})$, i.e.\
the radius of the current ball $k+1$ with the maximal radius in box
$J_{k+1}$:
\begin{itemize}
\item If $R_{k+1}>R^\ast(k,J_{k+1}) + \di n^{-1/d}$, then ball $k+1$ is
large enough to cover the ball that corresponds to the maximal radius
$R^\ast(k,J_{k+1})$, because the distance of their centres is at most
$\di n^{-1/d}$ (as the centres are in the same box). So, we may delete
from the collection the balls that correspond to the current maximal radius
$R^\ast(k,J_{k+1})$ (it {\it is} in the collection by property 2)) and add
the $(k+1)$-th ball. At the same time, $R_{k+1}$ becomes the maximal radius, i.e.
$R^\ast(k+1,J_{k+1})=R_{k+1}$. Certainly, this preserves 1), as any ball
that we deleted from the collection is covered by the ball that was added to
the collection. It also preserves property 2), since the newly added ball
is the one that corresponds to the maximal radius now.
\item If $R_{k+1}<R^\ast(k,J_{k+1}) - \di n^{-1/d}$, then the $(k+1)$-th
ball is not added to the collection. Note that it is covered
by the ball that corresponds to the maximal radius, because the distance between
their centres is at most $\di n^{-1/d}$ (the centres being in the same box).
This shows that property 1) is preserved, and certainly 2) is preserved, because
the maximal radius is unchanged, $R^\ast(k+1,J_{k+1})=R^\ast(k,J_{k+1})$.
\item Finally, if
$R_{k+1} \in [R^\ast(k,J_{k+1})-\di n^{-1/d},R^\ast(k,J_{k+1})+\di n^{-1/d}]$,
we do add the $(k+1)$-th ball into the collection. It may have the new maximal radius
or not, in both cases (since we do not exclude any ball from the collection) properties
1) and 2) are preserved.
\end{itemize}

We now count how often the last of the three cases above occurs. Let us define
the corresponding event
$$
  A_{k+1}:=\left\lbrace N(k,J_{k+1})>0, R_{k+1}
  \in \left[R^\ast(k,J_{k+1})-\frac{\di}{n^{1/d}},R^\ast(k,J_{k+1})+\frac{\di}{ n^{1/d}}\right]
  \right\rbrace
$$
and denote by
$$
   S_m:= \sum_{k=1}^m \ind_{A_k}
$$
the number of occurrences of the last of the three cases up to $m$ steps
of the algorithm. Let us further denote by $K_c(m)$ the size of the
collection after $m$ steps, and set
$$
   K'(m):= m - \sum_{k=1}^m \ind_{ \{ \exists i \neq k : B(\xi_k,R_k)\subseteq B(\xi_i,R_i)\} }
$$
for the number of balls (with index $\leq m$) that are not covered by {\it some} other ball.
The major observation is that, because of property 1) and the fact that in each step of the
algorithm the number of balls increased by at most one (either because we had $N(k,J_{k+1})=0$
or because $A_{k+1}$ occured), we have, respectively,
\begin{equation} \label{eqn:secondestimate}
   K'(m)\leq K_c(m) \leq n + S_m.
\end{equation}

Another important observation is that for $K$ from the statement of the proposition and the
Poisson random variable  $N$ we have
\begin{equation} \label{eqn:firstestimate}
   K \leq K'(N).
\end{equation}

{\it Step 3:} Evaluation of the probability of one $A$-event given the past.

Let the $\sigma$-fields $\FF_k$ be defined by
$$
   \FF_k := \sigma( (\xi_i)_{i\leq k+1}; (R_i)_{i\leq k}), \qquad k=0,1,2,\ldots.
$$
These $\sigma$-fields represent the information about the centres and radii up
to step $k$.

We show that there is a number $\kappa>0$ (only depending on the law of the radii)
such that for each $k$
\begin{equation} \label{eqn:aevent}
   p_{k+1} := \P[ A_{k+1} | \FF_k] \leq \kappa \, n^{-\alpha/d},
\end{equation}

To see (\ref{eqn:aevent}), note that
\begin{eqnarray*}
    p_{k+1} &:=&  \P[ A_{k+1} | \FF_k]
\\
    &=& \P[ R_{k+1} \in
    [R^\ast(k,J_{k+1})-\di n^{-1/d},R^\ast(k,J_{k+1})+\di n^{-1/d}]
    \big| \FF_k ].
\end{eqnarray*}
Note that $R^\ast(k,J_{k+1})$ is measurable w.r.t.\ $\FF_k$ and
$R_{k+1}$ is independent of $\FF_k$. Therefore, we can estimate
$$
   p_{k+1} \leq \sup_{x>0} \P[ R_{k+1} \in [x-\di n^{-1/d},x+\di n^{-1/d}]
    \leq \di\, c\, n^{-\alpha/d} := \kappa \, n^{-\alpha/d},
$$
where we used the bound for the concentration function of $R_{k+1}$
from the proposition's  assumption.

{\it Step 4:} Majorization of $S_m$.


Set $\Delta=\Delta(n):=\kappa  n^{-\alpha/d}$.

We shall prove by induction that for any $\gamma>0$ and any $m\in\N$ we have
\begin{equation} \label{eqn:majorization1}
   \E \exp( \gamma S_m )\leq (\Delta (e^{\gamma}-1)  +1)^m.
\end{equation}

Clearly the estimate holds for $m=0$. Assume the estimate holds for $m=k$. Then
\begin{eqnarray*}
   \E \exp( \gamma S_{k+1} )
   &=& \E [ \E[ \exp( \gamma (S_{k} + \ind_{A_{k+1}})) | \FF_k ]]
\\
   &=& \E [ \exp( \gamma S_{k} ) \E[ \exp( \gamma \ind_{A_{k+1}}) | \FF_k ] ]
\\
   &=& \E [ \exp( \gamma S_{k} ) \E[ e^\gamma \ind_{A_{k+1}} + \ind_{A_{k+1}^c}
       \big| \FF_k ] ]
\\
   &=& \E[ \exp( \gamma S_{k} )  ( e^\gamma p_{k+1}  + (1-p_{k+1}))]
\\
   &=& \E[ \exp( \gamma S_{k} )  ((e^\gamma-1) p_{k+1}  + 1)]
\\
   &\leq & \E[ \exp( \gamma S_{k} )  ((e^\gamma-1) \Delta  + 1)]
\end{eqnarray*}
where we used (\ref{eqn:aevent}) in the last step. This shows the claim
in (\ref{eqn:majorization1}).

{\it Step 5:} Final computations.

Fix $B\in\N$, $B\geq 2$. Then due to (\ref{eqn:firstestimate}),
(\ref{eqn:secondestimate}), and (\ref{eqn:majorization1}) we have
for any $\gamma>0$
\begin{eqnarray*}
   \P[ K\geq B n ] &\leq& \P[ K'(N) \geq B n]
\\
   &=& \sum_{m=0}^\infty \P[ N=m] \cdot \P[ K'(m) \geq B n]
\\
   &\leq & \sum_{m=0}^\infty \P[ N=m]  \cdot  \P[ S_m \geq (B-1) n]
\\
   &\leq & \sum_{m=0}^\infty \P[ N=m] \cdot  \E[ \exp( \gamma S_m ) ]
   \, e^{-\gamma (B-1)n}
\\
   &\leq & \sum_{m=0}^\infty \frac{\lambda^m}{m!} e^{-\lambda} \cdot
   \left( \Delta (e^\gamma-1)  + 1\right)^m \, e^{-\gamma (B-1)n}
\\
   &= & e^{-\lambda-\gamma (B-1)n}\, \sum_{m=0}^\infty \frac{\lambda^m}{m!}
   \cdot  \left( \Delta (e^\gamma-1)  + 1\right)^m
\\
   &= & \exp\left( -\gamma (B-1)n + \lambda  \Delta (e^\gamma-1)  \right).
\end{eqnarray*}

Choosing $\gamma:=\log ((B-1)n/(\lambda\Delta))$ (and $n$ large enough to
ensure that $\gamma>0$), we obtain the estimate
$$
   \log \P[ K\geq B n ] \leq -(B-1) n \log ((B-1)n/(\lambda\Delta))
   + (B-1) n - \lambda\Delta.
$$

Since $(B-1)n/\Delta\sim  (B-1) n^{1+\alpha/d}/\kappa$, dividing the last
estimate by $B n\log (B n)$ and letting $n\to \infty$ gives
$$
    \limsup_{n\to \infty} \frac{\log \P[ K\geq n]}{ n \log n}
    = \limsup_{n\to \infty} \frac{\log \P[ K\geq B n]}{B n \log (Bn)}
    \leq - \frac{B-1}{B} ( 1+ \alpha/d).
$$
Letting $B\to \infty$ shows the claim.
\end{proof}

\begin{rem}
We did not use the fact that the $(\xi_i)$ are uniformly distributed,
i.e.\ any distribution on $[0,1]^d$ works.
\end{rem}

\begin{rem}
The same argument works if the $(\xi_i)$ take values in a totally bounded
metric space: Let us denote the covering numbers of that space by $N(\eps)$,
$\eps>0$. Let
$$
   Q_R(r):=\sup_{x>0} \P\big[ R_1 \in [x,x+r]\big]
$$
be the concentration function of $R$. Then the above proof can be modified
to show that for any $\eps>0$
$$
    K(m) \leq N(\eps) + S_m
$$
where $S_m$ is a random variable satisfying
$$
    \E e^{\gamma S_m} \leq ( Q_R(\eps) (e^\gamma -1) + 1)^m,
    \qquad \forall m\in\N ~ \forall \gamma>0
$$
and
$$
     K(m):= \min\{ r\geq 1 | \exists i_1,\ldots, i_r \in \{1,\ldots,m\} :
     \bigcup_{s=1}^m B(\xi_{i_s},R_{i_s}) = \bigcup_{s=1}^r B(\xi_{i_s},R_{i_s}) \},
$$
\end{rem}

\subsection{Constant radius for $\ell_1$-, $\ell_2$-norms} \label{sec:constantupper}

\subsubsection{Constant radius: $\ell_1$-norm}

The following result completely matches the lower bound from Proposition
\ref{prop:lowerc1}.

\begin{prop} \label{prop:upperc1}
Assume that the radius is a.s.\ constant $R_1\equiv c<1$ and we work
with $\ell_1$-balls. If $d\geq2$, then
\begin{equation} \label{eqn:constl1upper}
    \P[ K \geq n ] \leq \exp\left( - \left(1+\frac{1}{d-1}\right) n \log n (1+o(1)) \right),
    \qquad \text{as $n\to\infty$.}
\end{equation}
\end{prop}

\begin{proof}
The first three steps of the proof are deterministic ones,
while probability estimates appear in the fourth step.
\medskip

{\sl Step 1. Combining the balls in groups.}\
For a while, let the norm be arbitrary and let
\[
   S=\bigcup_{i=1}^K B(\theta_i,r_i)\ \cap \ [0,1]^d
\]
be an irreducible representation of the picture $S$. Then for every $i\leq K$
there exists a point $\nu_i\in B(\theta_i,r_i) \cap \ [0,1]^d$ such that
$\nu_i\not\in B(\theta_j,r_j)$ for $j\not=i$. Let $\Delta_i:=\nu_i-\theta_i$.

Let $r:=\min_{1\leq i\leq K} r_i$ and denote
$J_0:= \{i: ||\Delta_i|| \leq r/2\}$. Then for any distinct $i,j\in J_0$ we have
\[
   r\leq r_i< ||\nu_j- \theta_i|| \leq ||\nu_j-\nu_i|| + ||\nu_i- \theta_i||
   =  ||\nu_j-\nu_i||  + ||\Delta_i||  \leq  ||\nu_j-\nu_i||  + r/2.
\]
It follows that $||\nu_j-\nu_i|| > r/2$ and we conclude that
$\#\,J_0\leq c_1\, r^{-d}$ with a constant $c_1:=2^d / \Vol B(0,1)$ depending
only on the dimension and the norm.

Let now $i\not\in J_0$. Then
\[
   \max_{1 \leq m  \leq d} |\Delta_i^{(m)}| =||\Delta_i||_\infty \geq c_2 ||\Delta_i||
   > c_2\, r/2
\]
with $c_2$ depending only on the norm. Therefore, $i$ belongs to one of the $2d$ sets
\[
    J_m^+:=\{  i: \ \Delta_i^{(m)}> c_2\, r/2  \},
    \qquad
    J_m^-:=\{  i: \ \Delta_i^{(m)}< - c_2\, r/2  \}.
\]
From now on we fix one of these sets, say, $J_d^+$. For  $i\in J_d^+$ we have
\[
    \nu_i^{(d)}  - \theta_i^{(d)} = \Delta_i^{(d)}\geq c_2\, r/2.
\]
\medskip

{\sl Step 2. Evaluation of coordinate differences.}
We specify to the case of equal radii $r_1=\cdots=r_K=r$ and $\ell_1$-norm;
the subsequent constants $c_j$ are allowed to depend on $r$.

At this step we give a bound for the difference
$|\theta_i^{(d)}-\theta_j^{(d)}|$ for $i,j\in J_d^+$
and show that it can be either quite large or small.

Let $\sigma:\R^d\mapsto \R^{d}$ be the projection defined by
\[
  \sigma x:=(x^{(1)},\cdots, x^{(d-1)},0).
\]

\begin{lem} \label{lem:thetad1}
Let  $i,j\in J_d^+$ and $c_3:= c_2\, r/2$. Then
\be
   |\theta_i^{(d)}-\theta_j^{(d)}|
   \not\in [\, ||\sigma\theta_i-\sigma\theta_j||_1,  c_3 \, ].
\ee
\end{lem}

\begin{proof} [ of Lemma] Without loss of generality we may and will assume that
$i\not=j$ and $\theta_i^{(d)}>\theta_j^{(d)}$. Introduce an auxiliary point
$\psi:=(\theta_i^{(1)},\cdots,\theta_i^{(d-1)},\theta_j^{(d)})$.
We have
\begin{eqnarray} \notag
  ||\nu_j-\psi||_1
  &\le&  ||\nu_j-\theta_j||_1+ ||\theta_j-\psi||_1
\\   \notag
  &=& ||\nu_j-\theta_j||_1 + ||\sigma\theta_j-\sigma\theta_i||_1
\\  \label{eqn:up1}
  &\le& r + ||\sigma\theta_j-\sigma\theta_i||_1.
\end{eqnarray}

Assume, for the sake of contradiction, that
\be\label{eqn:up2}
   \theta_i^{(d)}\in
   [\, \theta_j^{(d)}+ ||\sigma\theta_i-\sigma\theta_j||_1,
     \theta_j^{(d)}+  c_3 \, ].
\ee
Since $j\in J_d^+$, \eqref{eqn:up2} yields
\[
   \nu_j^{(d)}\geq\theta_j^{(d)} + c_2 r /2 = \theta_j^{(d)} +c_3
   \geq\theta_i^{(d)}.
\]
On the other hand,
\be \label{eqn:up3}
   ||\nu_j-\psi||_1 =  ||\sigma\nu_j-\sigma\psi||_1 + \nu_j^{(d)}-\psi^{(d)}
   =||\sigma\nu_j-\sigma\theta_i||_1 + \nu_j^{(d)}-\theta_j^{(d)}.
\ee
Therefore, by using \eqref{eqn:up3},  \eqref{eqn:up2},  \eqref{eqn:up1}  we obtain
\begin{eqnarray*}
    ||\nu_j-\theta_i||_1
    &=& ||\sigma\nu_j -\sigma\theta_i||_1 + \nu_j^{(d)}- \theta_i^{(d)}
\\
    &=&  ||\nu_j-\psi||_1 -  (\nu_j^{(d)}-\theta_j^{(d)})  + \nu_j^{(d)}- \theta_i^{(d)}
\\
    &=&  ||\nu_j-\psi||_1   - (\theta_i^{(d)}  -\theta_j^{(d)})
\\
    &\le&  ||\nu_j-\psi||_1   - ||\sigma\theta_i-\sigma\theta_j||_1
\\
    &\le&   r + ||\sigma\theta_j-\sigma\theta_i||_1  - ||\sigma\theta_i-\sigma\theta_j||_1
    =r,
\end{eqnarray*}
which contradicts the assumption $\nu_j\not \in B(\theta_i,r)$ from
the definition of $\nu_j$. It follows that \eqref{eqn:up2} does not hold,
hence, the assertion of Lemma \ref{lem:thetad1} is true.
\end{proof}
\medskip

{\sl Step 3. Counting boxes' hits.}

Let us fix a large $A>0$ and cover $[0,1]^d$ with the following
collection of cubic boxes:
$$
    V_{\bark,k_d}:= \prod_{m=1}^d
    \left[ \frac{A k_m}{n^{1/(d-1)}},
    \frac{A(k_m+1)}{n^{1/(d-1)}}\right],
$$
with $k_m\in\{0,\ldots,\lfloor  A^{-1} n^{1/(d-1)}\rfloor\}$ for
$1 \leq m \leq d$ and multi-index $\bark:=(k_1,\ldots, k_{d-1})$.

Let us fix $\bark$, and evaluate the number of the corresponding
boxes hit by the ball centers:
\[
   N(\bark,d,+):= \# \{ k: \theta_i\in  V_{\bark,k} \textrm{for some } i\in J_d^+\}.
\]
Indeed, if $\theta_i\in  V_{\bark,\kappa_i}$ and  $\theta_j\in  V_{\bark,\kappa_j}$
for some $i,j\in J_d^+$,  then
\[
    ||\sigma\theta_i-\sigma \theta_j||_1 \leq (d-1)A n^{-1/(d-1)}
\]
and Lemma \ref{lem:thetad1} yields
\[
   |\theta_i^{(d)}-\theta_j^{(d)}|
   \not\in [\, (d-1)A n^{-1/(d-1)}  , c_3\, ].
\]
By splitting $[0,1]$ into $\lceil c_3^{-1} \rceil$ pieces of length less or equal to $c_3$
we see that if $\theta_i^{(d)},\theta_j^{(d)}$ belong to the same piece, then
$|\theta_i^{(d)}-\theta_j^{(d)}| \leq (d-1)A n^{-1/(d-1)}$. Hence,
$|\kappa_i-\kappa_j| \leq d$.
It follows immediately that
\[
    N(\bark,d,+) \leq  d \lceil c_3^{-1} \rceil =: c_4.
\]
By the symmetry of coordinates, the total number of hit boxes admits the estimate
\[
   \sum_{\bark}\sum_{m=1}^{d}  \left( N(\bark,m,+) +  N(\bark,m,-) \right)
    \leq   (2d c_4) \cdot  (n/A^{d-1})=:\frac{c_5 n}{A^{d-1}}.
\]

Let denote $\UU$ the ensemble of all possible unions of  $\lfloor\frac{c_5 n}{A^{d-1}}\rfloor$ boxes.

Notice that $\#\UU$,  the number of choices of $\lfloor\frac{c_5 n}{A^{d-1}}\rfloor$ boxes
from the total number of $\frac{n^{d/(d-1)}}{A^{d}}$ boxes, admits the bound
\begin{eqnarray} \notag
   \#\UU &\leq&
   \exp\left(\frac{c_5 n}{A^{d-1}} \ln\left( \frac{n^{d/(d-1)}}{A^{d}}\right) \right)
\\  \label{eqn:UU}
   &=& \exp\left(    \frac{c_5 d n}{A^{d-1}(d-1)} \ln n (1+o(1)) \right).
\end{eqnarray}
For every $U\in\UU$ we have the bound
\be\label{eqn:volU}
   \Vol(U)  \leq \frac{c_5 n}{A^{d-1}} \cdot \left(\frac{A}{n^{1/(d-1)}} \right)^{d}
   = \frac{A c_5 }{n^{1/(d-1)}}.
\ee
\medskip

{\sl Step 4. Probabilistic estimates.}

Recall that
\[
   K=\#J_0+  \#\left(\bigcup_{m=1}^{d}(J_d^+ \cup J_d^-)\right)
   =: K^{(0)} + K^{(\pm)}.
\]

Notice that $K^{(\pm)}$ centers simultaneously belong to some random $U\in \UU$, which
can be written as
\[
   N_U:= \#\{i:\xi_i\in U\} \geq  K^{(\pm)}.
\]

Recall that $K^{(0)} \leq c_1\, r^{-d} =:c_6$.
Therefore,  we obtain
\begin{eqnarray*}
\P[K \geq n] &\leq& \P[ K^{(\pm)} \geq n - c_6\, ]
\\
&\leq& \sum_{U\in\UU}  \P[ N_U \geq n - c_6 ]
\\
&\leq& \#\UU \cdot   \max_{U\in\UU} \P[ N_U \geq n - c_6 ].
\end{eqnarray*}
Since for every deterministic $U$ the random variable $N_U$ is a
Poissonian one  with expectation
$\lambda \Vol(U)$,
by using \eqref{eqn:UU}, \eqref{eqn:volU} we have
\begin{eqnarray*}
\P[K \geq n]  &\leq&  \exp\left(    \frac{c_5 d n}{A^{d-1}(d-1)} \ln n (1+o(1)) \right)
           \left(\frac{\lambda\Vol(U)e}{n-c_6}\right)^{n-c_6}
\\
 &\leq&   \exp\left(    \frac{c_5 d n}{A^{d-1}(d-1)} \ln n (1+o(1)) \right)
           \left(\frac{(\lambda A c_5\, e) \,  n^{-1/(d-1)}}{n-c_6}
           \right)^{n-c_6}
\\
 &\leq&   \exp\left(   \left( \frac{c_5 d }{A^{d-1}(d-1)}
            -\left(1+ \frac{1}{d-1}\right)
           \right)  n \ln n (1+o(1)) \right).
\end{eqnarray*}
Since $A$ can be chosen arbitrarily large, we get (\ref{eqn:constl1upper}).
\end{proof}

\subsubsection{Constant radius: $\ell_2$-norm}

The following result corresponds to the lower bound from Proposition
\ref{prop:lowerc2} but does not exactly match it.

\begin{prop} \label{prop:upperc2}
Assume that the radius is a.s.\ constant $R_1\equiv r <1$ and we work
with $\ell_2$-balls. If $d \geq 2$, then
$$
    \P[ K \geq n ] \leq
    \exp\left( - \left(1+\frac{1}{d-1}\right) n \log n (1+o(1)) \right),
    \qquad \text{as $n\to\infty$.}
$$
\end{prop}

\begin{proof}
 Steps 1 and 4 of the proof are exactly the same as in the proof
of Proposition~\ref{prop:upperc1}. Step~3 is almost identical, up to an appropriate
modification of the constant $c_4$. We do have to modify Step~2, where
the particular form of the norm is used. The following is an appropriate modification
of Lemma~\ref{lem:thetad1}.

\begin{lem} \label{lem:thetad2}
Let  $i,j\in J_d^+$, $c_3:= c_2\, r/2$, $c'_3:= (2r+\sqrt{d-1})/c_3$. Then
\be
   |\theta_i^{(d)}-\theta_j^{(d)}|
   \not\in [\, c'_3 ||\sigma\theta_i-\sigma\theta_j||_2,  c_3 \, ].
\ee
\end{lem}

\begin{proof} [ of Lemma] Without loss of generality we may and do assume that
$i\not=j$ and $\theta_i^{(d)}>\theta_j^{(d)}$. Introduce the same auxiliary point
$\psi:=(\theta_i^{(1)},\cdots,\theta_i^{(d-1)},\theta_j^{(d)})$ as before.
We have
\begin{eqnarray} \notag
  ||\nu_j-\psi||_2
  &\le&  ||\nu_j-\theta_j||_2+ ||\theta_j-\psi||_2
 \\   \notag
  &=& ||\nu_j-\theta_j||_2 + ||\sigma\theta_j-\sigma\theta_i||_2
 \\ \label{eqn:up21}
  &\leq& r + ||\sigma\theta_j-\sigma\theta_i||_2.
\end{eqnarray}

Assume temporarily that
\be\label{eqn:up22}
   \theta_i^{(d)}\in
   [\, \theta_j^{(d)}+  c'_3 ||\sigma\theta_i-\sigma\theta_j||_2,
     \theta_j^{(d)}+  c_3 \, ].
\ee
Since $j\in J_d^+$, \eqref{eqn:up22} yields
\be  \label{eqn:up22a}
   \nu_j^{(d)}\geq \theta_j^{(d)} + c_2 r /2 = \theta_j^{(d)} +c_3 \geq \theta_i^{(d)}.
\ee
On the other hand,
\be \label{eqn:up23}
   ||\nu_j-\psi||_2^2 =  ||\sigma\nu_j-\sigma\psi||_2^2 + (\nu_j^{(d)}-\psi^{(d)})^2
   =||\sigma\nu_j-\sigma\theta_i||_2^2 + (\nu_j^{(d)}-\theta_j^{(d)})^2.
\ee
Therefore, by using \eqref{eqn:up23}, \eqref{eqn:up21}, \eqref{eqn:up22}, \eqref{eqn:up22a}
and the definition of $c'_3$, we obtain
\begin{eqnarray*}
    ||\nu_j-\theta_i||_2^2
    &=& ||\sigma\nu_j -\sigma\theta_i||_2^2 + (\nu_j^{(d)}- \theta_i^{(d)})^2
\\
    &=&  ||\nu_j-\psi||_2^2 -  (\nu_j^{(d)}-\theta_j^{(d)})^2  + (\nu_j^{(d)}- \theta_i^{(d)})^2
\\
    &=&  ||\nu_j-\psi||_2^2   - (\theta_i^{(d)}
         -\theta_j^{(d)}) (2\nu_j^{(d)}- \theta_i^{(d)}- \theta_j^{(d)} )
\\
    &\le& \left(r + ||\sigma\theta_j-\sigma\theta_i||_2\right)^2
         - (\theta_i^{(d)}  -\theta_j^{(d)})
           (\nu_j^{(d)}-  \theta_j^{(d)} )
\\
    &\le&  r^2 +  (2r+\sqrt{d-1}) ||\sigma\theta_j-\sigma\theta_i||_2
            - c'_3 ||\sigma\theta_i-\sigma\theta_j||_2 \cdot c_3
\\
    &=& r^2,
\end{eqnarray*}
which contradicts the assumption $\nu_j\not \in B(\theta_i,r)$ from
the definition of $\nu_j$. It follows that \eqref{eqn:up22} does not hold,
hence, the assertion of Lemma \ref{lem:thetad2} is true.
\end{proof}

The rest of the proof of Proposition \ref{prop:upperc2} goes exactly as
 in the $\ell_1$-norm case.
\end{proof}

\section{Dimension $d=1$}  \label{sec:dimension1}

 \begin{proof}[ of Theorem~\ref{thm:d1}]
 The lower bound follows from Proposition \ref{prop:lowerSB} which is valid
 for any dimension. The upper bound is based on the following elementary lemma.

 \begin{lem} \label{lem:length}
 Let
 \be \label{eqn:Krepr}
    S=\bigcup_{i=1}^K \ [x_i-r_i,x_i+r_i]\ \cap \ [0,1]
 \ee
 be an irreducible representation of a one-dimensional picture $S$. Then
 \be \label{eqn:Klength}
    \sum_{i=1}^K \min\{r_i,1\}  \leq 2.
 \ee
 \end{lem}

 \begin{proof}[ of Lemma \ref{lem:length}]
 We notice first that any point $x\in[0,1]$ is covered by at most two intervals
 $[x_i-r_i,x_i+r_i]$. The corresponding indices are those where
 $\min_{i: x\in[x_i-r_i,x_i+r_i]} (x_i-r_i)$ and
 $\max_{i: x\in[x_i-r_i,x_i+r_i]} (x_i+r_i)$ are attained.
 Would there be another interval covering $x$, it would be covered by those two
 we have chosen which contradicts to irreducibility of \eqref{eqn:Krepr}.

Using that for any $x\in[0,1],r>0$
\[
   \Volone\left([x-r,x+r]\cap[0,1]\right)  \geq \min\{r,1\},
 \]
we have by integration
 \begin{eqnarray*}
    \sum_{i=1}^K \min\{r_i,1\} &\leq&  \sum_{i=1}^K
    \Volone\left([x_i-r_i,x_i+r_i]\cap[0,1]\right)
 \\
    &=& \sum_{i=1}^K \int_{0}^{1} \ed{x\in[x_i-r_i,x_i+r_i]} dx
 \\
 &=& \int_{0}^{1}\left( \sum_{i=1}^K  \ed{x\in[x_i-r_i,x_i+r_i]}\right) dx
 \leq \int_{0}^{1} 2 \, dx =2,
 \end{eqnarray*}
 and \eqref{eqn:Klength} follows.
 \end{proof}

 We derive now the upper bound in \eqref{eqn:d1}. Fix a large $M>0$ and observe that
 the identity
 \[
    K=\#\{i: r_i>2M/n\} + \#\{i: r_i \leq 2M/n\} =: K_1+K_2.
 \]
 The bound \eqref{eqn:Klength} ensures that $K_1 \leq n/M$ for $n \geq 2M$,
 while $K_2$
 is bounded by the total number of balls of radius less or equal to $2M/n$
 in the {\it initial} representation of the picture $S$. The latter is
 a Poissonian random variable with expectation
 $a:= \lambda \int_0^{2M/n} p(z)dz  \leq c\, (M/n)^\alpha$.
 Therefore,
\begin{eqnarray*}
    \P[K\geq n] &\leq& \P[K_2\geq (1-1/M)n]
\\
    &\le& e^{-a} \frac{a^{(1-1/M)n}}{[(1-1/M)n]!}\, (1+o(1))
\\
    &\sim&
    \frac{a^{(1-1/M)n} e^{(1-1/M)n}}{\sqrt{2\pi (1-1/M) n}[(1-1/M)n]^{(1-1/M)n}}
\\
   &\le&  \frac{[ c\, (M/n)^\alpha]^{(1-1/M)n} e^{(1-1/M)n}}{[(1-1/M)n]^{(1-1/M)n}}
\\
    &=& \exp( - (1+\alpha)(1-1/M) n \log n\, (1+o(1)) ).
\end{eqnarray*}
Letting $M\to\infty$ proves the required upper bound.
\end{proof}

\begin{rem} If $\P[R_i\geq r]=1$ for some $r>0$, then it follows from
\eqref{eqn:Klength} that $\P[K\leq \tfrac{2}{\min\{r,1\}}]=1$.
 This means that if the radii are separated from zero, the large
deviations for $K$ are trivial. \label{rem:empty}
\end{rem}

\section{The coding problem}  \label{sec:quantization}

\subsection{Proof of Theorem~\ref{thm:d1quantization}}
Notice that $m$ in the theorem is the maximal possible number of disjoint intervals that
compose our random set $S$.

\begin{proof} Let us start with the case $c<1/2$.

{\it Upper bound.} Let $\eps \in (0,1)$ and let $G_\eps:=\{\eps,2\eps...,
\lfloor \eps^{-1}\rfloor\eps \}$
be the corresponding grid.
We produce a dictionary
\[
   \CC:= \left\{ \bigcup_{j=1}^k [a_j,b_j], \ 1\le k\le m, a_j\le b_j, a_j,b_j\in
G_\eps
   \right\} \cup \{\emptyset\}.
\]

If our random set $S$ is not empty, it is a union of $k$ intervals with  $1\le k\le m$
and we always have a set $C\in \CC$ such that $d_H(S,C)\le \eps$.

On the other hand, we have
\[
 \#\CC  \le \sum_{k= 1}^m \lfloor\eps^{-1}\rfloor^{2k} \le   m \, \eps^{-2m}.
\]
For given large $r$, we choose $\eps$ from equation $m\,\eps^{-2m}=\exp(r)$, i.e.
$\eps:= m^{1/2m} e^{-r/2m}$. We conclude with the required bound
\[
  D^{(q)}(r) \leq  m^{1/2m} e^{-r/2m} :=c_2 \, e^{-r/2m}.
\]
\medskip

{\it Lower bound.} By the definition of $m$ there exists a sufficiently small
$\delta$ such that
$m(2c+4\delta)<1$. Therefore, we may place  $m$ disjoint intervals
$[ z_k- (c+2\delta), z_k+(c+2\delta)]$, $1\le k\le m$, into $[0,1]$. In the
following, we will consider the case when
$S$ is a union of pairwise overlapping $2m$ intervals of length $2c$ with centers
$\xi_{2k-1}\in [z_k-\delta,z_k]$ and $\xi_{2k}\in [z_k,z_k+\delta]$, for $1\le k\le
m$.
We have
\[
    S = \bigcup_{k=1}^m [a_{2k-1}, a_{2k}],
\]
where
\begin{eqnarray*}
    a_{2k-1} &=& \xi_{2k-1} -c \in [z_k- c-\delta,z_k-c],
\\
    a_{2k} &=& \xi_{2k} +c \in [z_k+ c,z_k+c +\delta].
\end{eqnarray*}

Moreover, the random points $a_k$, $1\le k\le 2m$, are independently
and uniformly distributed on the corresponding
intervals. We will denote $a:=(a_k)_{1\le k\le 2m}$ the corresponding
$2m$-dimensional random vector uniformly distributed on a cube of side
length $\delta$.

Let us now fix a non-random closed set $C\subset [0,1]$ and  build
a $2m$-dimensional deterministic vector $b=b(C):=(b_k)_{1\le k\le 2m}$ by
\begin{eqnarray*}
    b_{2k-1} &=& \min \{x\big| x\in C\cap [z_k- (c+2\delta),z_k+(c+2\delta)],
\\
    b_{2k} &=& \max \{x\big| x\in C\cap [z_k- (c+2\delta),z_k+(c+2\delta)].
\end{eqnarray*}
Assume first that all sets $C\cap [z_k- (c+2\delta),z_k+(c+2\delta)]$
are non-empty, thus $B$ is well defined. The main observation is as follows:
we have
\[
   d_H(S,C)\ge \min \{ ||a-b||_\infty; \delta\} \ge \delta \cdot ||a-b||_\infty,
\]
where the minimum with $\delta$ appears because of possible points in $C$
outside each interval $[z_k-(c-2\delta),z_k+(c+2\delta)]$. Similarly, if the set
$C\cap [z_k- (c+2\delta),z_k+(c+2\delta)]$ is empty
for some $k\le m$, we simply have
\[
   d_H(S,C) \ge \delta  \ge  \delta \cdot ||a-b||_\infty.
\]
For any dictionary $\CC$ we thus have
\[
   \E \min_{C\in \CC} d_H(S,C) \ge  \delta \ \E \min_{C\in \CC} ||a-b(C)||_\infty.
\]
By using the well known bound for the quantization error of finite-dimensional
vectors uniformly distributed on cubes (see e.g.\ Lemma~22 in \cite{vormoor} or \cite{lp}), it follows immediately that
\[
   D^{(q)}(r) \ge \delta\,  D^{(q)}(a,\|.\|_\infty;r)\ge  c_1 \, e^{-r/2m},
\]
where $c_1$ depends on $m$ and on $\delta$.

{\it The case $c\geq 1/2$.} For this case, the result is the same as in \eqref{d1c} with
$2m$ replaced by $1$. The reason is that in this case $S$ consists of a unique
interval with only {\it one} random end.
 \end{proof}

\subsection{Proof of the upper bound in Theorem~\ref{thm:coding1}}
The upper bound follows from the next general lemma, which relates an upper bound
for the asymptotics of $\P[K\geq n]$ for $n\to\infty$ to the upper bound for
the quantization error.

\begin{lem}  \label{lem:quantuppergeneral}
Assume that for some $a>0$
$$
   \P[K\geq n] \leq \exp( - a \cdot n \log n (1+o(1)) ),
   \qquad \text{as $n\to\infty$}.
$$
Then
\be \label{lem19_gen}
  D^{(q)}(r)  \leq \exp\left( -\sqrt{ \frac{2 a}{d+1}\, r \log r } \cdot (1+o(1))\right),
   \qquad \text{as $r\to\infty$}.
\ee
Moreover, if the radius is constant, then
\be \label{lem19_Rconst}
   D^{(q)}(r)  \leq \exp\left( -\sqrt{ \frac{2 a}{d}\, r \log r } \cdot (1+o(1))\right),
   \qquad \text{as $r\to\infty$}.
\ee
\end{lem}

\begin{proof}[ of Lemma~\ref{lem:quantuppergeneral}]
The proof relies on the following coding strategy. Recall that $K$ is the number of balls
needed in order to produce the random picture $S$ {\it without any error}. We shall encode
the positions and radii of these $K$ balls {\it approximately} and thus retrieve the picture
$S$ approximately. In particular, if $K=k$,
then our random picture admits a representation
\[
 S= \bigcup_{i=1}^k B(\txi_i,\tR_i)
\]
and one has to encode a vector
in $[0,1]^{d k}\times \R_+^k$ (the centres $\txi_i$, $i=1,\ldots,k$ and the radii
$\tR_i$, $i=1,\ldots,k$). It is clear that one does not have to encode radii that are larger
than the diameter $\di$ of the unite cube  $[0,1]^d$, as in that case the picture is
trivial, i.e. $S=[0,1]^d$.

Note that Lemma~22 in \cite{vormoor} (also see \cite{lp}) gives an explicit bound on the
quantization error in $[0,1]^\ell$ with respect to $\ell_\infty$-norm, namely,
for any random element
$X\in [0,1]^\ell $ and any $\rho>0$ there exists a dictionary $\CC(\ell,\rho)$ such that
$\# \CC(\ell,\rho) \leq e^{\rho}$ and
\be \label{cubequant}
    \E \min_{x\in \CC(\ell,\rho)} || X - x ||_\infty \leq e^{-\rho/\ell}.
\ee

Fix $k$ with $1\leq k\leq r$.
Using \eqref{cubequant} with $\ell=(d+1)k$, $\rho:=r-k$,  we can choose sets
$\CC_k=\{ (p_i^j,r_i^j)_{i=1,\ldots,k},j=1,\ldots,\# \CC_k \}$ with $\# \CC_k \leq e^{r-k}$
such that
\be \label{cenradquant}
 \E \left[ \min_j || (\txi_i,\tR_i/\di)_{i=1}^k
                     - (p_i^j,r_i^j/\di)_{i=1}^k ||_\infty
    \Big| K=k \right]
 \leq  e^{-(r-k)/(k(d+1))}.
\ee
Furthermore, we build a sub-dictionary of pictures,
\[
  \CC^\circ_k= \left\{  \bigcup_{i=1}^k  B(p_i^j,r_i^j), j=1,\ldots,\# \CC_k
               \right\}.
\]
Set $\CC^\circ_0:=\{[0,1]^d\}$ and define the full dictionary by
$\CC^\circ:=\bigcup_{k=0}^r \CC^\circ_k$.
Then $\#\CC^\circ \leq 1 + \sum_{k=1}^r e^{r-k} \leq e^r$ for large enough $r$.

On the other hand, we may use the following result.
\begin{lem} \label{lem:dHBB}
  For any norm $||.||$ on $\R^d$, the respective
balls $B(\cdot,\cdot)$ and the respective Hausdorff distance, we have
for all centers $x,x'\in \R^d$ and radii $r,r'>0$
\be \label{dHBB}
   d_H(B(x,r),B(x',r')) =  ||x-x'|| + |r-r'|.
\ee
\end{lem}

\begin{proof}[ of Lemma~\ref{lem:dHBB}]
The upper bound follows from the triangle inequality in Hausdorff distance,
\begin{eqnarray*}
   d_H(B(x,r),B(x',r')) &\leq&   d_H(B(x,r),B(x',r)) +  d_H(B(x',r),B(x',r'))
\\
   &\leq& ||x-x'|| + |r-r'|.
\end{eqnarray*}
The lower bound follows from the triangle inequality in $\R^d$.
Let $v\in B(x,r)$ be an arbitrary point and choose  $v'\in B(x',r')$
such that $x,x',v'$ lay on the same line, $||x'-v'||=r'$, and $x'$
is situated between $x$ and $v'$. Then
\begin{eqnarray*}
   r+||v-v'|| &\geq& ||x-v|| + ||v-v'|| \geq ||x-v'||
 \\
  &=& ||x-x'||+ ||x'-v'|| =  ||x-x'||+ r',
\end{eqnarray*}
which implies $||v-v'||\geq ||x-x'||+ r'-r$. Therefore,
\[
    d_H(B(x,r),B(x',r'))\geq \inf_{v\in B(x,r)} ||v-v'||\geq ||x-x'||+ r'-r.
\]
By the full symmetry, we obtain
\[
    d_H(B(x,r),B(x',r'))\geq ||x-x'||+ |r'-r|,
\]
as required.
\end{proof}

Furthermore, \eqref{dHBB} yields an upper estimate of the same type
for the unions of balls, namely,
\[
    d_H\left( \bigcup_{i=1}^k B(x_i,r_i),\bigcup_{i=1}^k B(x'_i,r'_i)\right)
    \leq \max_{1 \leq i \leq k}  [\, ||x_i-x_i'|| + |r_i-r_i'| \,] .
\]
Applying this estimate to any element of $\CC^\circ_k$ and to $S$,
we obtain
\begin{eqnarray*}
    d_H\left(  \bigcup_{i=1}^k  B(p_i^j,r_i^j ), S \right)
    &=&  d_H\left( \bigcup_{i=1}^k  B(p_i^j,r_i^j),  \bigcup_{i=1}^k B(\txi_i,\tR_i) \right)
\\
    &\leq& \max_{1 \leq i \leq k}  [\, ||\txi_i -p_i^j|| + |\tR_i -r_i^j| \,]
\\
    &\leq& \max_{1 \leq i \leq k}  \left[\,  d \di\, ||\txi_i-p_i^j||_\infty
    + \di\, \left|\frac{\tR_i}{\di} -\frac{r_i^j}{\di}\right| \,\right]
\\
     &\leq&  (d+1) \di \,
     || (\txi_i,\tR_i/\di)_{i=1}^k - (p_i^j,r_i^j/\di)_{i=1}^k ||_\infty.
\end{eqnarray*}

By \eqref{cenradquant}  we obtain
\begin{eqnarray*}
    && \E\left[ \min_{C\in \CC_k} d_H(C,S) \Big| K=k\right]
\\
    &\leq&   (d+1) \di \,
            \E \left[ \min_j
        || (\xi_i,R_i/\di)_{i=1}^k - (p_i^j,r^j/\di)_{i=1}^k ||_\infty
         \Big| K=k\right]
\\
    &\leq&  (d+1) \di \cdot e^{-(r-k)/(k(d+1))}.
\end{eqnarray*}

Thus, we obtain
\begin{eqnarray*}
    D^{(q)}(r) & \leq & \E \min_{C\in\CC} d_H(C,S)
\\
    &= & \sum_{k=0}^r \P[K=k] \cdot \E [\min_{C\in\CC} d_H(C,S) | K= k]
\\
    && + \sum_{k=r}^\infty \P[K=k] \cdot \E [\min_{C\in\CC} d_H(C,S) | K= k]
\\
    &\leq & \sum_{k=0}^r \P[K=k] \cdot \E [\min_{C\in\CC_k} d_H(C,S) | K= k]
    + \sum_{k=r+1}^\infty \P[K=k] \cdot 1
\\
    &\leq &   (d+1) \di \, \sum_{k=0}^r \P[K=k]  e^{-(r-k)/(k(d+1))}
\\
    && + e^{1/(d+1)}\sum_{k=r+1}^\infty \P[K=k] \cdot e^{-r/(k(d+1))}
\\
    &\leq &  ( (d+1) \di+e^{1/(d+1)})   \sum_{k=0}^\infty \P[K=k]  e^{-r/(k(d+1))}
\\
    &\leq & ( (d+1) \di+e^{1/(d+1)}) \, \E e^{-r/(K(d+1))}.
\end{eqnarray*}

By a standard Tauberian theorem (cf.\ \cite{bgt}, Theorem 4.12.9), the assumption
$\P[K\geq k]\leq \exp(-a k \log k (1+o(1)))$, as $k\to\infty$ of our lemma
gives the conclusion \eqref{lem19_gen}.
Moreover, if the radius is constant, the same reasoning applies with $\ell=dk$
instead of $\ell=(d+1)k$ and we arrive at \eqref{lem19_Rconst}.
\end{proof}

\subsection{Proof of the lower bound in Theorem~\ref{thm:coding1}} In the following,
$c$ is used for a constant not depending on $r$, $n$ that may change at each occurrence.

Consider the following collection of boxes:
$$
    \prod_{m=1}^d \left[ \frac{k_m}{(2n)^{1/d}}+\frac{1/4}{(2n)^{1/d}},
    \frac{k_m}{(2n)^{1/d}}+\frac{3/4}{(2n)^{1/d}}\right],
    \qquad k_m\in\{0,\ldots,\lfloor (2n)^{1/d}\rfloor-1\}.
$$
The number of boxes being of order $2n$, we may choose among them $n$ distinct boxes, say $V_1, \ldots, V_n$.
In the sequel we will also need larger boxes $W_1,\ldots,W_n$ from the collection
$$
   \prod_{m=1}^d \left[ \frac{k_m}{(2n)^{1/d}},
       \frac{k_m+1}{(2n)^{1/d}}\right],\qquad
       k_m\in\{0,\ldots,\lfloor (2n)^{1/d}\rfloor-1\}.
$$
such that $W_i\supseteq V_i$ for each $i$.

Define the following event:
$$
     E:=\bigcup_{\text{$\pi$ permutation of } \{1,\ldots,n\} }
     E_\pi,
$$
where
\[
  E_\pi := \left\{ N=n, \xi_i\in V_{\pi(i)},
     i=1,\ldots,n , R_i \in [c_1 \theta n^{-1/d}, c_2 \theta n^{-1/d}]
     \right\},
\]
$c_2:=2^{-3-1/d}$, $c_1:=c_2/2>0$, and $\theta>0$ is chosen so that
$||x||\geq \theta ||x||_\infty$ for all $x\in \R^d$.

In view of the lower bound in the assumption on the density $p$,
the probability of this event admits the following bound:
\begin{eqnarray}
\P[E]
     &=& \frac{\lambda^n}{n!} e^{-\lambda} \cdot n!
        \cdot  \left( \frac{1/2}{(2n)^{1/d}} \right)^{dn}
        \cdot \P[ R_1 \in [c_1 n^{-1/d}, c_2 n^{-1/d}] ]^n
        \notag
\\
    &=& (\lambda/2^{d+1})^n e^{-\lambda} n^{-n} \cdot
        \left(\int_{c_1 n^{-1/d}}^{c_2 n^{-1/d}} p(z) \dd z\right)^n
        \notag
\\
   &\geq& (\lambda/2^{d+1})^n e^{-\lambda} n^{-n}
   \cdot \left( c \, n^{-\alpha/d} \right)^n
   \notag
\\  \label{eqn:probabofE}
   &=& \exp( - (1+\alpha/d) n \log n \cdot(1+o(1))).
\end{eqnarray}

Consider
\begin{eqnarray*}
D^{(q)}(r) &=& \inf_{\#\CC\leq e^r} \E[ \min_{C\in\CC} d_H(C,S) ]
\\
&\geq& \inf_{\#\CC\leq e^r} \E[ \min_{C\in\CC} d_H(C,S) \ind_E ]
\\
&=& \inf_{\#\CC\leq e^r} \E[ \min_{C\in\CC} d_H(C,S) |E ] \cdot \P[E].
\end{eqnarray*}
Further, denoting by $\KKK$ the set of all measurable subsets of $[0,1]^d$
note that for any dictionary $\CC$ with $\#\CC\leq e^r$
and any $\delta>0$
\begin{eqnarray*}
   \E[ \min_{C\in\CC} d_H(C,S) |E ] &\geq& \delta \cdot
   \P[ \forall C\in\CC : d_H(C,S) \geq \delta |E]
\\
   &=& \delta \cdot (1- \P[ \exists C\in\CC : d_H(C,S) < \delta |E])
\\
   &\geq & \delta \cdot (1- \#\CC \cdot \sup_{C\in\KKK} \P[ d_H(C,S) < \delta |E]).
\\
   &\geq & \delta \cdot (1- e^r \cdot \sup_{C\in\KKK} \P[ d_H(C,S) < \delta |E])
\\
   &\geq & \delta \cdot (1- e^r \cdot \sup_{C\in\KKK}  \sup_\pi\,
   \P[ d_H(C,S) < \delta |E_\pi]).
\end{eqnarray*}

Combining this bound together with the last estimate yields
\be \label{eqn:Dqr}
     D^{(q)}(r) \geq \P[E] \cdot \delta
     \cdot (1- e^r \sup_{C\in\KKK}  \sup_\pi \P[ d_H(C,S) < \delta |E_\pi]).
\ee

Now we estimate $\P[ d_H(C,S) < \delta |E_\pi]$ for a fixed set $C$,
fixed permutation $\pi$ and
\be \label{eqn:delta}
   \delta< \frac{\theta}{8(2n)^{1/d}}.
\ee

We first show that under $E_\pi$ for each $i\leq n$ one has
\be \label{BinW}
    B(\xi_i,R_i+\delta) \subset W_{\pi(i)}.
\ee
Indeed, if $x\in B(\xi_i,R_i+\delta)$, then
\[
   ||x-\xi_i||_\infty \leq \theta^{-1} ||x-\xi_i|| \leq \theta^{-1} (R_i+\delta)
    \leq c_2 n^{-1/d}+\delta/\theta  \leq 2^{-2}(2n)^{1/d}.
\]
Since $\xi_i\in V_{\pi(i)}$, we obtain $x\in W_{\pi(i)}$ and \eqref{BinW} follows.

We see from \eqref{BinW} that all balls in the representation
\[
   S=\bigcup_{i=1}^n  B(\xi_i,R_i)
\]
are not only disjoint but $\delta$-separated. Therefore, if
$d_H(C,S)<\delta$, then  $C\cap W_{\pi(i)} \neq \emptyset$ for all $i$ and so
\[
    d_H(C\cap W_{\pi(i)},S \cap W_{\pi(i)})
     =  d_H(C\cap W_{\pi(i)},B(\xi_i,R_i)) < \delta.
\]
Since $C$ is deterministic, when $\pi$ is fixed there exists a deterministic ball
$B(x_i,r_i)$ such that $d_H(C\cap W_{\pi(i)},B(x_i,r_i)) < \delta$.

Indeed, let $U$ (here $U=C\cap W_{\pi(i)}$ for short) be a deterministic set
such that $\P[ d_H (U, B(x(\omega), r(\omega))) <\delta] > 0$.
Take a countable set of balls $(B(x_k,r_k))_{k\in \N}$ which is $d_H$-dense
in the set of all balls. We clearly have
\begin{eqnarray*}
  \sum_{k\in \N} \P[d_H (U, B(x_k, r_k)) <\delta] &\geq& \P( \inf_k  d_H (U, B(x_k, r_k)) <\delta)
\\
  &\geq& \P[ d_H (U, B(x(\omega), r(\omega))) <\delta] > 0.
\end{eqnarray*}
Obviously, there exists some $k\in\N$ such that  $\P[d_H (U, B(x_k, r_k)) <\delta]>0$.

But both sets, $U$ and $B(x_k,r_k)$, are deterministic. Therefore, we simply have
$d_H (U,  B(x_k, r_k) )<\delta$, as required.

Hence, $d_H(C,S)<\delta$ yields, by the triangle inequality,
\begin{multline*}
   d_H(B(x_i,r_i),B(\xi_i,R_i))
   \\ \leq d_H(B(x_i,r_i),C\cap W_{\pi(i)})+d_H(C\cap W_{\pi(i)},B(\xi_i,R_i))
   < 2\delta.
\end{multline*}

The equality \eqref{dHBB} yields now $||\xi_i-x_i|| \leq 2\delta$
and $||R_i-r_i|| \leq 2\delta$. Recall that $x_i,r_i$ are deterministic
and depend only on $C$ and $\pi$.

Even after conditioning on $E_\pi$, the ensembles of centers $(\xi_i)_{1 \leq i \leq n}$
and  radii $(R_i)_{1 \leq i \leq n}$ remain independent; while  $\xi_i$ is uniformly distributed
on $V_\pi(i)$ and  $R_i$ is distributed on $[c_1 \theta n^{-1/d}, c_2 \theta n^{-1/d}]$
with a density proportional to $p$.

These observations show that
\begin{eqnarray*}
  && \P\left(  d_H(S,C)<\delta \Big| E_\pi\right)
\\
  &\le&
    \P\left( ||\xi_i-x_i|| \leq 2\delta, |R_i-r_i|  \leq 2\delta, 1 \leq i \leq n  \Big| E_\pi\right)
\\
  &=& \prod_{i=1}^{n}  \P\left( ||\xi_i-x_i|| \leq 2\delta \Big| E_\pi\right)
  \cdot
   \prod_{i=1}^{n}  \P\left( |R_i-r_i|  \leq 2\delta \Big| E_\pi\right).
\end{eqnarray*}

We clearly have
\[
   \P\left( ||\xi_i-x_i|| \leq 2\delta \big| E_\pi\right)
   \leq \frac{\Vol(B(0,1))(2\delta)^d}{\Vol(V_1)} =: c\, \delta^d n.
\]
Using the upper bound in the assumption on the density $p$ we obtain
$$
    \P\left[ |R_i-r_i| < 2\delta \big| E_\pi \right]
     \leq c\, \int_{r_i-2\delta}^{r_i+2\delta} z^{\alpha-1} \dd z
        (c\, n^{-\alpha/d})^{-1}
     \leq c\, \delta n^{1/d}.
$$
Hence,
\be \label{eqn:condprob}
    \P\left(  d_H(S,C)<\delta \big| E_\pi\right) \leq
     (c\, \delta^d n)^n \cdot
   (c\, \delta n^{1/d})^n =: c^n\, \delta^{(d+1)n} \, n^{(1+1/d)n}.
\ee

Putting estimates \eqref{eqn:probabofE},\eqref{eqn:Dqr}, \eqref{eqn:condprob}
together yields
$$
    D^{(q)}(r) \geq \exp( - (1+\alpha/d) n \log n (1+o(1)) )
    \cdot \delta \cdot \left( 1 - e^{r}  c^{n}\, \delta^{(d+1)n}
     \, n^{(1+1/d)n} \right).
$$
Now we choose $\delta$ such that
$$
    e^{r} c^{n }\delta^{(d+1)n} n^{(1+1/d)n} = 1/2
$$
and obtain  $\delta = c\, e^{-r/((d+1)n)} n^{-1/d}$, which gives
\begin{eqnarray*}
   D^{(q)}(r)  &\geq& \exp\left( - (1+\alpha/d) n \log n (1+o(1)) -r/((d+1)n)\right)
\\
   &=:& \exp\left( - A n \log n (1+o(1)) -r/(Bn) \right).
\end{eqnarray*}
Now we optimize in $n$ by letting $n\sim \sqrt{\frac{2r}{AB \log r}}$
and obtain
\begin{eqnarray*}
   D^{(q)}(r)  &\geq& \exp\left(- \sqrt{\frac{2Ar \log r}{B}}\ (1+o(1)) \right)
\\
  &=& \exp\left(-\sqrt{\frac{2(1+\alpha/d)}{d+1}\,  r\log r}\ (1+o(1))\right),
\end{eqnarray*}
as required in the assertion of the theorem. It remains to notice that
the choice of $\delta$ agrees with required property  \eqref{eqn:delta}
for large $n$ and $r$.

\subsection{Proof of Theorem~\ref{thm:coding-2-3}, $\ell_1$-balls part}

The upper bound follows from the claim \eqref{lem19_Rconst}
of Lemma \ref{lem:quantuppergeneral}
where we may let $a:=\frac{d}{d-1}$ by Proposition \ref{prop:upperc1}.

For getting the lower bound we use the construction from the proof of
Proposition \ref{prop:lowerc1} and the proof scheme of the lower bound
in Theorem \ref{thm:coding1}. We repeat everything for completeness.

Consider the following collection of boxes:
$$
    \left\{
       \prod_{m=1}^{d-1}
       \left[ \frac{k_m}{(2n)^{1/(d-1)}}+\frac{1/4}{(2n)^{1/(d-1)}},
       \frac{k_m}{(2n)^{1/(d-1)}}+\frac{3/4}{(2n)^{1/(d-1)}}\right]
    \right\}
    \times \left[0,\frac{c_2}{n^{1/(d-1)}}\right],
$$
and the larger tubes
$$
    \left\{
        \prod_{m=1}^{d-1}
        \left[ \frac{k_m}{(2n)^{1/(d-1)}},
        \frac{k_m+1}{(2n)^{1/(d-1)}}\right]
    \right\}
  \times [0,1],
$$
with $k_m\in\{0,\ldots,\lfloor (2n)^{1/(d-1)}\rfloor-1\}$.
Here, $c_2:=2^{-( 4+1/(d-1))}$.

The number of boxes being of order $2n$, we may choose among them $n$ distinct boxes,
say $V_1, \ldots, V_n$ and use the corresponding tubes $U_1,\ldots, U_n$
such that $V_i\subset U_i$, $i=1,\dots,n$.

As before, we consider the event $E$,
$$
     E:=\bigcup_{\text{$\pi$ permutation of } \{1,\ldots,n\} } E_\pi,
$$
where
\[
  E_\pi := \left\{ N=n, \xi_i\in V_{\pi(i)}, i=1,\ldots, n \right\},
\]
and recall from (\ref{eqn:againsameprob}) the bound
\be \label{eqn:probabofE_c1}
    \P[E] \geq \exp\left( - \frac{d}{d-1}\  n \log n (1+o(1)) \right).
\ee

We will use inequality \eqref{eqn:Dqr} with this $E$ and these $E_\pi$.
Note that its derivation does not depend on the concrete event $E$, but it holds
for any event.

As in the previous proof, we have to estimate $\P[ d_H(C,S) < \delta |E_\pi]$
for a fixed set $C$, fixed permutation $\pi$ and small $\delta$, however
using very different geometric arguments.

Recall that notation $c$ is used for a constant not depending on $r$ or $n$
that may change at each occurrence.
Instead of \eqref{eqn:condprob}, we will prove
\be \label{eqn:condprob_1c}
    \P\left(  d_H(S,C)<\delta \big| E_\pi\right)
    \leq
     \left(c\, \delta^d n^{d/(d-1)}\right)^n.
\ee
Putting estimates \eqref{eqn:Dqr}, \eqref{eqn:probabofE_c1}, \eqref{eqn:condprob_1c}
together yields
$$
    D^{(q)}(r) \geq \exp\left( -  \frac{d}{d-1}\ n \log n (1+o(1)) \right)
    \cdot \delta \cdot \left( 1 - e^{r}   \left(c\, \delta^d n^{d/(d-1)}\right)^n  \right).
$$

Now we choose $\delta$ such that
$$
    e^{r}  \left(c\, \delta^d n^{d/(d-1)}\right)^n  = 1/2
$$
and obtain  $\delta = c\, 2^{-1/dn}\, e^{-r/(dn)} n^{-1/(d-1)}$, which gives
\begin{eqnarray*}
   D^{(q)}(r)  &\geq& \exp\left( -  \frac{d}{d-1}\  n \log n (1+o(1))-r/(dn)\right)
\\
   &=:& \exp\left( - A n \log n (1+o(1)) -r/(Bn) \right)
\end{eqnarray*}
with $A:=\frac{d}{d-1}$ and $B=d$. We optimize in $n$ as before, by letting
$n\sim \sqrt{\frac{2r}{AB \log r}}$ and obtain
\begin{eqnarray*}
   D^{(q)}(r)  &\geq& \exp\left(- \sqrt{\frac{2Ar \log r}{B}}\ (1+o(1)) \right)
\\
  &=& \exp\left(-\sqrt{\frac{2}{d-1}\, r\log r}\ (1+o(1))\right),
\end{eqnarray*}
as required in the assertion of the theorem.


It remains to prove \eqref{eqn:condprob_1c}. To this aim, we fix a deterministic
set $C$ and a permutation $\pi$.
Assume that
\be \label{eqn:d_H_delta}
  d_H(S,C)<\delta
\ee
with a small $\delta$ such that
\be \label{eqn:delta2}
    \delta< \frac{\theta}{2^6(2n)^{1/(d-1)}}.
\ee
For every $i\le n$ we have the following. Let
\[
  y_i := \textrm{argmax} \{ y^{(d)} | y\in C\cap U_{\pi(i)}\}
\]
be a local top point of $C$ and let
$x_i:=\xi_i+(0,\dots,0,R)$ be the top point of the ball $B(\xi_i,R)$.
We will show that $y_i$ and $x_i$ are close.

First, we prove that
\be \label{eqn:yx}
   y_i^{{(d)}}\ge  x_i^{{(d)}} -\delta.
\ee
Indeed, by \eqref{eqn:d_H_delta} there exists $y\in C$ such that
$||y-x_i||_1\le \delta$.
Using $\xi_i\in V_{\pi(i)}$ and the inequality
$\delta \le \frac{1/4}{(2n)^{1/(d-1)}}$,
we see that  $y\in U_{\pi(i)}$. Hence,
\[
   y_i^{{(d)}}\ge  y^{{(d)}} \ge  x_i^{{(d)}} -  ||y-x_i||_1 \ge  x_i^{{(d)}} -\delta.
\]
\medskip

Second, for any $b\in B(\xi_i,R)$ it is true that
\begin{eqnarray*}
  R-||x_i-b||_1 &\ge& ||b-\xi_i||_1- ||x_i-b||_1
  = |b^{{(d)}}-\xi_i^{{(d)}}| -(x_i^{{(d)}}- b^{{(d)}})
\\
   &\ge& (b^{{(d)}}-\xi_i^{{(d)}}) -(x_i^{{(d)}}- b^{{(d)}})
   = 2(b^{{(d)}}-x_i^{{(d)}})+R,
\end{eqnarray*}
hence,
\be \label{eqn:bx}
   b^{{(d)}}\le x_i^{{(d)}}- ||x_i-b||_1/2.
\ee
\medskip

Third, by (\ref{eqn:d_H_delta}) there exists a $b_i\in S$ such that
$||b_i-y_i||_1\le \delta$.
In particular,
\be  \label{eqn:yb}
   y_i^{{(d)}}  \le b_i^{{(d)}}+ \delta.
\ee
Moreover,  it is true  that $b_i\in B(\xi_i,R)$.
Indeed, assume that $b_i\in B(\xi_j, R)$ for some $j\not= i$.
Then
\begin{eqnarray*}
  b_i^{{(d)}} &\le& x_j^{{(d)}}- ||x_j-b_i||_1/2
  \le  x_i^{{(d)}} + \frac{c_2}{n^{1/(d-1)}}  - (||x_j-y_i||_1-\delta)/2
\\
  &\le&  x_i^{{(d)}} + \frac{c_2}{n^{1/(d-1)}}  - \frac{1}{8(2n)^{1/(d-1)}} +
  \delta/2
\\
   &=&  x_i^{{(d)}}  -  \frac{1}{2^4((2n)^{1/(d-1)}} + \delta/2
  <  x_i^{{(d)}} -2\delta.
\end{eqnarray*}
Here we used  inequality \eqref{eqn:bx} with $b=b_i$ and with $j$ instead of $i$,
the definition of $c_2$, and the bound \eqref{eqn:delta2} for $\delta$.
The result contradicts
\[
   b_i^{{(d)}}\ge  y_i^{{(d)}} -\delta \ge  x_i^{{(d)}} -2\delta
\]
and we see that $b_i\in B(\xi_j, R)$, for $j\not= i$, is impossible.
\medskip

Fourth,
by applying \eqref{eqn:bx} to $b=b_i$ and combining with \eqref{eqn:yb},
it follows that
\[
  y_i^{{(d)}}  \le x_i^{{(d)}}- ||x_i-b_i||_1/2 + \delta.
\]
 By comparing this inequality with \eqref{eqn:yx} we obtain
$||x_i-b_i||_1 \le 4\delta$,
 and so by the definition of $b_i$ we get $||x_i-y_i||_1 \le 5\delta$.
The latter is equivalent to
$||\xi_i-z_i||_1 \le 5\delta$, where a deterministic point $z_i$
is defined by $z_i:=y_i-(0,\dots,0,R)$.

These observations show that
\begin{multline*}
  \P\left(  d_H(S,C)<\delta \Big| E_\pi\right)
  \le
    \P\left( ||\xi_i-z_i||_1 \leq 5\delta, 1 \leq i \leq n  \Big| E_\pi\right)
 \\= \prod_{i=1}^{n}  \P\left( ||\xi_i-z_i||_1 \leq 5\delta \Big| E_\pi\right),
\end{multline*}
and using that $\xi_i$ is uniformly distributed in $V_{\pi(i)}$ on $E_\pi$ we get
\[
   \P\left( ||\xi_i-z_i||_1 \leq 5\delta \big| E_\pi\right)
   \leq \frac{\Vol(B(0,1))(5\delta)^d}{\Vol(V_1)} =: c\, \delta^d n^{d/(d-1)}.
\]
Hence,
\[
    \P\left(  d_H(S,C)<\delta \big| E_\pi\right)
    \leq \left(c\,\delta^{d} \, n^{d/(d-1)}\right)^n,
\]
as required in \eqref{eqn:condprob_1c}. \hfill $\Box$

\subsection{Proof of Theorem~\ref{thm:coding-2-3}, $\ell_2$-balls part}

This proof closely follows the previous one with two minor changes.
For getting the upper bound, we refer to Proposition \ref{prop:upperc2}
instead of Proposition \ref{prop:upperc1}.
For getting the lower bound we use the construction from the proof of
Proposition \ref{prop:lowerc2} instead of Proposition \ref{prop:lowerc1}.

Furthermore, the geometric properties of the $\ell_2$-norm come into play.
We must use inequality
\[
   b^{{(d)}}\le x_i^{{(d)}}- ||x_i-b||_2^2/(2R_1).
\]
instead of \eqref{eqn:bx}. All other arguments go through exactly as before.
\hfill $\Box$

\bigskip
{\bf Acknowledgement.} This research was supported by the Russian Foundation 
Basic Research grant 16-01-00258 and by the co-ordinated
grants of DFG (GO420/6-1) and St.\ Petersburg State University (6.65.37.2017).

\bibliographystyle{abbrv}

\noindent {\bf Addresses of the authors:}\\
Frank Aurzada, Technische Universit\"at Darmstadt, Schlossgartenstra\ss e 7, 64287 Darmstadt, Germany\\
Mikhail Lifshits, St.\ Petersburg State University, 197372 St.\ Petersburg, Postbox 104, Russia

\end{document}